\documentclass[10pt]{article}
\usepackage{enumerate}
\usepackage{amsmath}
\usepackage{amssymb}
\usepackage{amsthm}
\usepackage[mathscr]{eucal}

\def\endjump{\relax}
\long\def\jump#1\endjump{\relax}
\newif\ifdraft\draftfalse

\def\mactex{\advance\hoffset by .5in}

\font\crw=cmss10
\font\cyrr=wncyss10

\def\lmonth{\ifcase \month \or January \or February\or March
  \or April\or May\or June\or July\or August\or September\or
  October\or November\or December\fi}
\def\today{\the\day\ \lmonth\ \the\year}

\newcount\cno\cno=0
\def\endchap{\vfil\supereject\relax\immediate\write16{Ending chapter \the\cno}}
\long\def\newchap{\vfil\supereject\ifodd\pageno\else{{
\headline{}}}\fi
\resetchap\advance\cno by
1\immediate\write16{Beginning chapter \the\cno}}

\def\lowern{{\kern.2pt \rm{{\cyrr I}}}}

\def\inv{^{-1}}
\def\db<{\mathrel{<\!<}}
\def\dbk#1<{\mathrel{<\!<_{\K_{#1}}}}
\def\dbkp<{\mathrel{<\!<_{\K_p}}}

\def\kp<{\mathrel{<_{\K_p}}}
\def\k#1<{\mathrel{<_{\K_{#1}}}}
\def\nor{\mathrel{\triangleleft}}

\def\chr{\mathrel{\rm{char}}}

\def\crownstar{\mathop{\rm{\raise2truept
\def\epsq{{\epsilon(q)\!}}
\hbox{{\crw Y}}\kern-5.8truept\lower6.2truept\hbox{*\hskip1.96truept}}}}
\def\ov{\overline}

\def\wt{\widetilde}

\def\FM#1{{\cal F\cal M9}}
\def\F{{\bf F}}

\def\chev{Chev}

\def\td4#1{{^3}\!D_4(#1)}
\def\epsq{{\epsilon(q)\!}}

\def\cal{\Cal}
\def\gp<#1>{\left\langle#1\right\rangle}
\def\<#1>{\gp<#1>}

\def\itemitemitem"#1"{{\advance\parindent by 10pt
\itemitem"#1"\advance\parindent by -10pt}}
\def\|#1:#2|{|#1\!:\!#2|}

\def\Inn{\mathop{\rm {Inn}}}
\def\Out{{\rmstyle\mathop{\rm {Out}}}}

\def\Aut{{\mathop{\rm {Aut}}}}
\def\aut{{\rm{Aut}}}

\def\Syl{{\mathop{\rm {Syl}}}}

\def\ndiv{\mathrel{\hbox{$\vert$\kern-2.5pt$\prime$}}}

\def\ch#1/#2/{{{{#1}\choose{#2}}}}
\newif\ifmargnotes\margnotesfalse
\def\marginalnotes{\margnotestrue
\gdef\std{\dp\strutbox}
\gdef\xmrg##1{\strut\vadjust{\kern-\std\vtop to
\std{\baselineskip\std \vss\llap{{\eightpoint ##1\ \ \ }}\null}}}
\gdef\ymrg##1{\strut\vadjust{\kern-\std\vtop to
\std{\baselineskip\std \vss\rlap{{\hskip\hsize\ \ \ \ \ \ \eightpoint ##1}}\null}}}
\gdef\zmrg##1{\strut\vadjust{\kern-\std\vtop to
\std{\baselineskip\std \vss\rlap{{\hskip\hsize\hskip1in \eightpoint ##1}}\null}}}
\gdef\mrg ##1{##1\ymrg{##1}}
\gdef\pmrg##1{##1\xmrg{##1}}}
\def\rf[#1;#2]{\ifmargnotes\mrg{[#1;#2]}\else[{\bf #1}; {#2}]\fi}
\def\rflem[#1;#2]{\ifmargnotes[#1;#2]\ymrg{#1:
{\tt\ignorespaces#2}}\else[{\bf #1};#2]\fi}

\def\({{\rm(}}
\def\){{\rm)}}
\def\[{{\rm[}}
\def\]{{\rm]}}
\def\quote`#1'{{\rm``}#1{\rm''}}

\newdimen\bht\bht0pt\newdimen\bwd\bwd0pt\newdimen\oht\oht0pt\newdimen\owd\owd0pt
\def\boxit#1{\setbox105\hbox{#1}\bht=\ht105\bwd=\wd105%
\advance\bwd by 4pt\advance\bht by 2.5pt%
\hbox to \bwd{\vtop {%
\hrule width \bwd height 0.1pt depth0pt \vss\hbox to \bwd{%
\vrule height \bht width 0.1pt depth4pt \hss\box105\hss
\vrule height \bht width 0.1pt depth4pt} \vss\hrule width \bwd height
0.1pt depth0pt}}}

\newdimen\tskp
\newdimen\ptskp
\newdimen\trnglsk
\newdimen\smspc
\newdimen\lgspc
\newdimen\thk
\newdimen\hlf
\newdimen\hfln
\hlf=\baselineskip\divide\hlf by 2\advance\hlf by \lgspc
\multiply \hlf by -1

\def\tablerule{\noalign{\hrule}}
\def\heavyrule{\noalign{\hrule height \thk }}

\def\spc#1{\omit&height#1 \double  \cr}

\def\lin#1{&&#1&\cr}
\newdimen\hlflinskip\hlflinskip=0pt
\def\hlflinskip{6.2truept}

\def\cl{\spc{\lgspc}}

\def\up{\heavyrule\cl\lin}
\def\dn{\cl\heavyrule}

\def\rone{\cl\tablerule\cl\lin}
\def\hone{\cl\heavyrule\cl\lin}

\newdimen\dispht\dispht=1truein
\newdimen\dispwd\dispwd=4.45truein
\newdimen\dispskip\dispskip=\baselineskip
\newcount\mmm\newcount\hhh
\def\tdig#1{\ifnum#1<10 0\fi\number#1}
\def\draft{\ifdraft\else
\mmm=\time\hhh=0
\loop\advance\hhh by 1 \advance\mmm by -60\ifnum\mmm>59\repeat
\advance\mmm by 60 \advance \hhh by -1

\headline=
  {\ifodd\pageno{\eightpoint\jobname---DRAFT---\today}\hfil{\tenrm\folio}
\else{\tenrm\folio}\hfil{\eightpoint\today---DRAFT---\jobname}\fi}
\fi\drafttrue}

\def\citek[#1]{\def\temp{#1}\def\sub{{\bf III$_{K2}$}}
          \ifx\temp\empty[\sub]\else[\sub, #1]\fi}
\def\citeg[#1]{\def\temp{#1}\def\sub{{\bf III$_7$}}
          \ifx\temp\empty[\sub]\else[\sub, #1]\fi}
\def\citei[#1]{\def\temp{#1}\def\sub{{\bf III$_8$}}
          \ifx\temp\empty[\sub]\else[\sub, #1]\fi}
\def\citen[#1]{\def\temp{#1}\def\sub{{\bf III$_9$}}
          \ifx\temp\empty[\sub]\else[\sub, #1]\fi}
\def\citex[#1]{\def\temp{#1}\def\sub{{\bf III$_{10}$}}
          \ifx\temp\empty[\sub]\else[\sub, #1]\fi}
\def\citexi[#1]{\def\temp{#1}\def\sub{{\bf III$_{11}$}}
          \ifx\temp\empty[\sub]\else[\sub, #1]\fi}
\def\citexii[#1]{\def\temp{#1}\def\sub{{\bf III$_{12}$}}
          \ifx\temp\empty[\sub]\else[\sub, #1]\fi}

\def\Q{{\bf Q}}

\newtheorem{theorem}{Theorem}[section]
\newtheorem{lemma}[theorem]{Lemma}

\newtheorem{corollary}[theorem]{Corollary}

\theoremstyle{definition}

\theoremstyle{remark}

\newcounter{overflow}
\newcount\temp\temp0
\newenvironment{myeqn}[0]
{\ifnum\value{equation}=26
\setcounter{overflow}{1}\setcounter{equation}{0}\fi
\ifnum\value{overflow}=1
\fi
\begin{equation}}
{\end{equation}\ignorespacesafterend}

\def\cal{\mathcal}

\def\roster{\begin{enumerate}[\ \ \ \ \rm(a)]}
 \def\endroster{\end{enumerate}}
\def\subroster{\begin{enumerate}[\rm(1)]}
 
\def\subsubroster{\begin{enumerate}[\greek(1)]}
  
\def\disproster\label#1{\begin{disptxt}{#1}\begin{enumerate}[\rm(1)]}
\def\dispsubroster{\begin{enumerate}[\rm(a)]}
\def\dispsubsubroster{\begin{enumerate}[\romnum(1)]}
\def\enddispsubsubroster{\end{enumerate}}
\def\enddispsubroster{\end{enumerate}}
\def\enddisproster{\end{enumerate}\end{disptxt}}
\def\disptext\label#1{\begin{disptxt}{#1}}
\def\enddisptext{\end{disptxt}}

\newdimen\boxwd\boxwd0pt\newdimen\boxht\boxht0pt\newdimen\boxdp\boxdp0pt
\def\boxit#1{\setbox3\hbox{#1}\boxwd=\wd3\advance\boxwd by 7pt%
\leavevmode\hbox{%
  \vrule width .5pt\vbox{\hsize=\boxwd%
  \hrule height .5pt\vskip5pt\hbox{%
  \hskip4pt{#1}\hskip5pt}%
  \vskip4pt\hrule height .5pt}%
  \vrule width .5pt}}
\newenvironment{exnote}[0]{\noindent{\sl Existence note:}\ }{}
\def\Aut{{\mathop{\rm Aut}}}
\def\aut{{\mathop{\rm Aut}}}
\def\Out{{\mathop{\rm Out}}}
\def\Inn{{\mathop{\rm Inn}}}
\def\Syl{{\mathop{\rm Syl}}}
\newcount\gpno\gpno0
\def\grp#1[#2]{\penalty1000000\global\advance\gpno by 1%
\bigskip\noindent\kern-.2in{\hbox{\boxit{\hbox{
          ${\bf G\cong#1}$}\quad\cite{#2}}}}\penalty10000}
\long\def\finaldisplay#1{\vskip6pt\hbox to \hsize{\hfil#1\hfil\hskip-\hsize\rightline{\qed}}}
\font\bbf=cmbx17
\font\smc=cmcsc10
\font\sm=cmr8
\font\smtt=cmtt8
\begin{document}

\centerline{{\bbf Automorphism groups of sporadic groups}}

\bigskip

\centerline{{\smc Richard Lyons}}
\smallskip
\centerline{{\sm Department of Mathematics, Rutgers University}}

\vskip-1pt
\centerline{{\sm 110 Frelinghuysen Road}}
\vskip-1pt

\centerline{{\sm Piscataway, NJ, 08854-8019 U.S.A.}}
\smallskip
\centerline{{{ \smtt lyons{}{\char64}{}math.rutgers.edu}} }

\medskip
\centerline{{\smc \today}}

\bigskip
\font\sm=cmr8

{\sm Among the simplest invariants of the sporadic finite
  simple groups are their outer automorphism groups. For 12
  of the 26 possible isomorphism types of a sporadic simple
  group G, the outer automorphism group Out(G) has order 2,
  and in the remaining 14 cases, Out(G) is trivial.
  Historically the suspicion of the existence of a sporadic
  group was followed in fairly short order by the
  calculation of a good upper bound on the size of its outer
  automorphism group. In a few cases establishing the
  existence of certain outer automorphisms, like the
  existence of the groups themselves, presented difficulties
  overcome only with the use of machine computation. In any
  case the calculations of the upper bounds, though
  typically straightforward, can be difficult to track down
  in the literature -- perhaps impossible in some
  cases. This note, which contains nothing new, is only
  intended to bring together these calculations. The
  proximate cause of writing them down was a question from
  Bob Oliver about the automorphism groups of some of these
  groups, how they were -- or might be -- calculated, and
  specifically whether the Sylow 2-subgroups of a sporadic
  simple group are self-centralizing in the automorphism
  group of the simple group. The answer is that they are.  }

\section{Introduction}

As each of the $21$ twentieth-century sporadic finite simple
groups $G$ came into view, the early properties that were
worked out included an upper bound on $|\Aut(G)|$, which
eventually was proved to be sharp. By comparison with some
other invariants -- for example, complete local structure,
maximal subgroups, and Schur multiplier -- the structure of
$\Aut(G)$ is relatively easy to settle, except for some
nagging cases where one proves that $|\Out(G)|\le 2$ but
equality is a computational challenge. We gather
calculations here for the record.  In particular they give
an affirmative answer to a question of Bob Oliver about
sporadic $G$'s: for $T\in\Syl_2(G)$, is $Z(T)$ a Sylow
subgroup of $C_{\aut(G)}(T)$?  \medskip

All groups $G$, $X$, etc., are to be assumed to be finite. 
We use the following notation: 

\def\tm{\smallskip\noindent\hangindent.4in}

\tm$G$ is a finite simple group, identified with its
  group $\Inn(G)$ of inner automorphisms. 

\tm{$A=\Aut(G)$ and $\wt A=A/G=\Out(G)$}

\tm{$\wt C(X)=C_A(X)G/G$ and $\wt N(X)=N_A(X)G/G$, for any $X\subseteq
G$} 

\tm{$T\in\Syl_2(G)$, $Z=Z(T)$, and $C=C_G(Z(T))$}

\tm{$M$ is a (large) proper subgroup of $G$ such that $T_M:=T\cap
M\in\Syl_2(M)$. A good choice of $M$ reveals much about
$\Aut(G)$.}

\tm{$F=F^*(M)$}

\tm{$\Aut_A(B)$ is the
  image of the natural mapping $N_A(B)\to\Aut(B)$ defined by
  $g\mapsto\,$conjugation by $g$; thus $\Aut_A(B)\cong
  N_A(B)/C_A(B)$}

\tm{$m_p(X)$ is the $p$-rank of the finite group $X$.}

\begin{theorem} Let $G$ be one of the sporadic groups listed in
Table $1$. Then $\wt A=\Out(G)$ is of order $1$ or $2$, as
indicated there. $G$ possesses a subgroup $M$ of the given type. If
$\wt A\ne \wt 1$, then the next column identifies $\wt A$ and $\wt 1$
in terms of $M$, and the final column gives a $2$-subgroup $S$ of $G$
such that $\wt C(S)=1$.
\end{theorem}

\begin{figure}
\centerline{{\bf TABLE 1}}
\def\lgspc{2.7pt}
\medskip
\def\double{&&&&&&&&&&&&}\def\jus{}
\def\thn{1pt}\def\thk{1.5pt}
\centerline{\vbox{\offinterlineskip\halign{%
\strut#&\tabskip0pt\vrule #width\thk\tabskip.35em&
\hss#\hss&\vrule #width\thn&
\hss#\hss&\vrule #width\thn&\hss#\hss&\vrule #width\thn&
\hss#\hss&\vrule #width\thn&\hss#\hss&\vrule #width\thn&
\hss#\hss&\vrule #width\thk\tabskip0pt\cr
\up{Pg.&&$G$&&$\Out(G)$&&$M$&&$\wt A/\wt 1$&&$2$-groups}
\hone{6&&$M_{11}$&&$1$&&$M_9\cong F_{9.8}$&&&&}
\rone{6&&$M_{12}$&&$Z_2$&&$M_{11}$&&$\wt A/\wt N(M)$&&$\wt C(T)=1$}
\rone{7&&$M_{22}$&&$Z_2$&&$M_{21}\cong L_3(4)$&&$\wt N(M)/\wt C(M)$&&$\wt
C(T_M)=1$}
\rone{8&&$M_{23}$&&$1$&&$M_{22}$&&&&}
\rone{8&&$M_{24}$&&$1$&&$M_{23}$&&&&}
\rone{10&&$J_1$&&$1$&&$N_G(T)$&&&&}
\rone{10&&$J_2=HJ$&&$Z_2$&&$C=2^{1+4}_-A_5$&&$\wt N(M)/\wt C(M)$&&$\wt C(O_2(C))=1$}
\rone{10&&$J_3$&&$Z_2$&&$C=2^{1+4}_-A_5$&&$\wt N(M)/\wt C(M)$&&$\wt C(O_2(C))=1$}
\rone{11&&$J_4$&&$1$&&$[C,C]=2^{1+12}_+3M_{22}$&&&&}
\rone{11&&$Co_1$&&$1$&&&&&&}
\rone{12&&$Co_2$&&$1$&&$C=2^{1+8}_+Sp_6(2)$&&&&}
\rone{12&&$Co_3$&&$1$&&$C=2Sp_6(2)$&&&&}
\rone{9&&$HS$&&$Z_2$&&$M_{22}$&&$\wt N(M)/\wt C(M)$&&$\wt C(T)=1$}
\rone{13&&$Mc$&&$Z_2$&&$C=2A_8$&&$\wt N(M)/\wt C(M)$&&$\wt C(T)=1$}
\rone{14&&$Suz$&&$Z_2$&&$G_2(4)$&&$\wt N(M)/\wt C(M)$&&$\wt C(T_M)=1$}
\rone{15&&$He=H\!H\!M$&&$Z_2$&&$E_{5^2}$&&$\wt N(M)/\wt C(M)$&&$\wt C(T)=1$}
\rone{15&&$Ly$&&$1$&&$3Mc2$&&&&}
\rone{16&&$Ru$&&$1$&&${^2}F_4(2)$&&&&}
\rone{13&&$O'N$&&$Z_2$&&$C=4L_3(4)2$&&$\wt N(M)/\wt C(M)$&&$\wt C(T)=1$}
\rone{17&&$Fi_{22}$&&$Z_2$&&$Z_2\times U_6(2)$&&$\wt N(M)/\wt C(M)$&&$\wt
C(T)=1$}
\rone{17&&$Fi_{23}$&&$1$&&$2Fi_{22}$&&&&}
\rone{18&&$Fi_{24}'$&&$Z_2$&&$Fi_{23}$&&$\wt C(M)$&&$\wt C(T)=1$}
\rone{19&&$F_5=HN$&&$Z_2$&&$A_{12}$&&$\wt N(M)/\wt C(M)$&&$\wt C(T)=1$}
\rone{20&&$F_3=Th$&&$1$&&$C=2^{1+8}_+A_9$&&&&}
\rone{21&&$F_2=BM$&&$1$&&$2\,{^2}E_6(2)$&&&&}
\rone{22&&$F_1=M$&&$1$&&$C=2^{1+24}_+Co_1$&&&&}
\dn
}}}
\end{figure}

We shall prove that $\Out(G)$ has order at most what is listed, and
cite the constructions of nontrivial outer automorphisms. 
 Occasionally we get an inductive benefit from our treating all
twenty-six groups together.

We make use of the charts of local structure of the sporadic
groups found in \cite{I:A}, mainly for centralizers of
involutions but occasionally for elements of other small
prime orders. The calculations for that information can be
made without any use of upper bounds on the size of the
automorphism group.

\section{Useful Folklore}

\begin{lemma}
\label{lemmaone}
Let $X$ be a group and set $Q=F^*(X)$.  Then
the following conditions hold:
\begin{enumerate}[{\rm(a)}]
\item $\pi(C_{\aut(X)}(Q))\subseteq \pi(F(X))$.
\item If $Q$ is a $p$-group and $P\in\Syl_p(G)$, then
$C_{\Aut(X)}(P)=\Aut_{Z(P)}(X)$.
\item Suppose that $Q$ is an abelian $p$-group and
$H^1(X/Q,Q)=1$. Then we have $C_{\aut(X)}(Q)=\Aut_Q(X)$.
\item Suppose that $Q$ is an extraspecial $p$-group and
$X/QX^{(\infty)}$ is a $p'$-group. Then
$C_{\aut(X)}(Q/Z(Q))=\Aut_Q(X)$ and $C_{\aut(X)}(Q)=\Aut_{Z(Q)}(X)$.
\item If $Q$ is an abelian or extraspecial $p$-group
and the image of $\Aut_X(Q)$ in $\Out(Q)$ is a self-normalizing
subgroup of $\Out(Q)$, then
$$
\Aut(X)=\Inn(X)\big(C_{\aut(X)}(X/Q)\cap C_{\aut(X)}(Q)\big).
$$
\end{enumerate}
\end{lemma}

\begin{proof} In each case we choose some $\alpha\in \Aut(X)$ and argue
in the semidirect product
$$
H=X\<\alpha>.
$$
In (a), let $\alpha$ be any $p$-element of $C_X(F^*(X))$ for
some prime $p$, and define $W=\<\alpha>Z(F(X))$. Then $W$ is
abelian.  By the $F^*$-Theorem, $W=\<\alpha
C_X(F^*(X))>=C_H(F^*(X))\nor H$. Hence $WF^*(X)=F^*(H)$ and
$W=Z(F^*(H))$. As $p$ does not divide $|F(X)|$,
$\<\alpha>=O_p(W)\nor H$, so $\alpha=1$ as an automorphism
of $X$.

For (b) and (c), set $Q=O_p(X)$ and choose any $\alpha$ such that
$[\alpha, Q]=1$. Then $C_H(Q)=Z(Q)\<\alpha>\nor H$. In particular
$[\alpha,X]\le C_X(Q)=Z(Q)$.  The mapping $f:X\to Z(Q)$ taking
$x\mapsto[x,\alpha]$ is a $1$-cocycle in $Z^1(X,Z(Q))$.  Then $f$ is
cohomologically trivial. This holds by hypothesis in (c), and in (b)
because $f$ vanishes identically on $P$ and the restriction mapping
$H^1(X,Z(Q))\to H^1(P,Z(Q))$ is injective, $P$ being a Sylow
$p$-subgroup of $X$.  Therefore there is $z\in Z(Q)$ such that
$[x,\alpha]=[x,z]$ for all $x\in X$. Then $\alpha$ is conjugation by
$x$. In (b), the hypothesis implies that $x\in Z(P)$. Thus (b) and (c)
hold.

For (d), choose any $\alpha$ such that $[\alpha,Q]\le Z(Q)$ and set $\ov
H=H/Z(Q)$. By (a), $\alpha$ is a $p$-element. Since $Q$ is
extraspecial, $\alpha$ induces an inner automorphism on $Q$, and hence
$C_H(Q)$ contains an element $\beta\in Q\alpha$. We set
$$
Z_1=\<\beta>Z(Q)=\<\beta>C_X(Q)=C_H(Q)\nor H.
$$
We have $[X,Z_1]\le X\cap Z_1=Z(Q)$. Let $Y=X^{(\infty)}$. Then
$[Y,Z_1,Y]\le [Z(Q),Y]=1$ so $[Y,Z_1]=1$. Therefore $C_H(Z_1)$
contains $QY$, so by assumption $H/C_H(Z_1)$ is a $p'$-group. Hence
either $Z_1$ is cyclic or we may write $Z_1=Z(Q)\times\<\gamma>$ with
$\<\gamma>\nor H$, for some $\gamma\in Z(Q)\beta\subseteq Q\alpha$. In
the latter case $\<\gamma>\cap X=1$, so $\gamma\in Z(H)$; in the
former case $H$ centralizes $Z_1/\Phi(Z_1)$ so centralizes $Z_1$. Thus
in any case $\gamma\in Z(H)$, so $\alpha\in QZ(H)$ and the result holds.

Finally in (e), let $A=\Aut(X)$. We have $X/Z(Q)\nor N_A(X)/Z(Q)$, so
the self-normalizing hypothesis forces $N_A(X)=XC_A(Q)$, and
$[C_A(Q),X]\le C_X(Q)\le Q$. The result follows.  
\end{proof}

\begin{corollary}
\label{charptype}
Suppose that $G$ is simple of characteristic $p$-type. 
Let $P\in\Syl_p(G)$. If $\alpha\in C_{\aut(G)}(P)$ and $\alpha$ is a nontrivial
$p'$-element, then $C_G(\alpha)$ is strongly $p$-embedded in
$G$. If $p=2$, then $C_{\aut(G)}(P)\le P$. 
\end{corollary} 

\begin{proof} For any $1\ne R\le P$ such $N_P(R)\in\Syl_p(N_G(R))$,
  $[\alpha,R]=1$ and so $\alpha$ normalizes $N_G(R)$, centralizing the
  Sylow $p$-subgroup $N_P(R)$. Hence $[\alpha, N_G(R)]=1$ by Lemma
  \ref{lemmaone}a. This has two consequences: first, normalizers of
  every extremal subgroup of $P$ lie in $C_G(\alpha)$; second, by
  Alperin's Theorem, $C_G(\alpha)$ controls strong $G$-fusion in
  $R$. Together these imply that $C_G(\alpha)$ contains $N_G(R)$ for
  all $1\ne R\le P$, proving the first statement. If $p=2$, then by
  the Bender-Suzuki theorem, $G$ is a Bender group. In this case it is
  well-known that $C_{\aut(G)}(P)\le P$. 
\end{proof}

\begin{lemma}
\label{lemmatwo}
Let $G$ be a simple group and
$T\in\Syl_2(G)$. Suppose that
\begin{enumerate}[\rm(a)]
\item $Z:=Z(T)=\<z>\cong Z_2$.
\item $F^*(C_G(Z))$ is a $2$-group.
\item There is $U\nor T$ such that $U\cong E_{2^2}$ and
$U^\#\subseteq z^G$.
\end{enumerate}
Let $\alpha\in C_A(T)$. Then one of the groups $W=C_G(\alpha)$,
$W=C_G(\alpha z)$ contains $\<C_G(Z),N_G(U)>$ and satisfies the conditions
\begin{enumerate}[\rm(1)]
\item $O_{2'}(W)=Z(W)=1$, and
\item $C_G(Z)<W$. 
\end{enumerate}

Moreover if $C_G(Z)$ is a maximal subgroup of $G$, or if $G$ is
characterized among all finite groups by the conditions $(1)$ and
$(2)$ and the isomorphism type of $C_G(Z)$, then
$C_{\aut(G)}(T)=\Aut_Z(G)\le \Inn(G)$.
\end{lemma}

\begin{proof} The importance of such groups $U$ was made clear in the
$N$-group paper of Thompson \cite{Thompson:NgpsI}. Set $T_0=C_T(U)$, so that
$|T:T_0|=2$ as $U\not\le Z(T)$. If $m_2(Z(T_0))>2$, then
$C_{Z(T_0)}(T/T_0)$ is noncyclic and lies in $Z$, which is a
contradiction. So $m_2(Z(T_0))=2$ and $U=\Omega_1(Z(T_0))\chr T_0$.
For any $u\in U^\#$, expand $T_0$ to $T_u\in\Syl_2(C_G(u))$. Since
$u\in z^G$, $T_u\in\Syl_2(G)$ and $\Aut_{T_u}(U)$ contains the unique
involution of $\Aut(U)$ fixing $u$. But $u$ was arbitrary, so
$\Aut_G(U)\cong \Aut(U)$. By a Frattini argument, there is a
$3$-element of $N_G(T_U)$ acting nontrivially on $U$. So
$N_G(T_U)\not\le C_G(Z)$.

Since $F^*(C_G(Z))$ is a $2$-group and $Z\le U$,
$F^*(C_G(U))=F^*(C_{C_G(Z)}(U))$ is a $2$-group as well (5.12 of
\cite{I:A}). By Lemma \ref{lemmaone}b applied to both $X=C_G(Z)$ and
$X=N_G(T_U)$, $\alpha$ centralizes $C_G(Z)$ and acts on $N_G(T_0)$ as
conjugation by $1$ or $z$. Replacing $\alpha$ by $\alpha z$ if
necessary, we may assume that $[\alpha, N_G(T_U)]=1$, whence
$$
W:=C_G(\alpha)>C_G(Z).
$$
If $O_{2'}(W)\ne 1$, then $Y:=C_{O_{2'}(W)}(u)\ne 1$ for some $u\in
U^\#$. Conjugating by an element of $N_G(U)$, we get
$C_{O_{2'}(W)}(z)\ne 1$. But $C_{O_{2'}(W)}(z)\le
O_{2'}(C_W(z))=O_{2'}(C_G(Z))=1$, contradiction. Thus
$O_{2'}(W)=1$. Finally if $Z(W)\ne 1$, it follows that $Z(W)\cap T\ne
1$, so $Z(T)\le Z(W)$ and $W\le C_G(Z)$, again a contradiction.
Therefore $Z(W)=1$, completing the proof of $(1)$ and $(2)$. The final
statement is an immediate consequence. 
\end{proof}

\begin{lemma}
\label{lemmathree}
Suppose that $G$ acts on the set
$\Omega$, and let $\alpha\in\Omega$. Suppose that 
\begin{enumerate}[\rm(a)]
\item $G$ acts faithfully and primitively on $\Omega$;
\item $G_\alpha$ acts faithfully and primitively on some
$G_\alpha$-orbit $\Psi$ on $\Omega-\{\alpha\}$;
\item $\Psi$ is the unique $G_\alpha$-orbit of length $|\Psi|$; 
\item One of the following holds:
\begin{enumerate}[\rm(1)]
\item Conditions (b) and (c) hold for all $G_\alpha$-orbits on
$\Omega-\{\alpha\}$; or 
\item For all $\beta\in\Psi$, $C_{\aut(G_\alpha)}(G_{\alpha\beta})=1$; and
\end{enumerate}
\item $G_\alpha$ is not cyclic.
\end{enumerate}
Then $C_{\aut(G)}(G_\alpha)=1$. 
\end{lemma}

\begin{proof} Let $a\in C_{\aut(G)}(G_\alpha)$, and set $H=G\<a>$, the
semidirect product. Then $H$ acts on $\Omega$ with
$H_\alpha=G_\alpha\times\<a>$. By (c), $a$ stabilizes $\Psi$. 
Suppose that $a\downarrow_\Psi\ne1$. By (b), $G_\alpha$ is faithful and
primitive on $\Psi$. Since $[a,G_\alpha]=1$, it follows that $\<a>$ is
of prime order and transitive on $\Psi$. But then $G_\alpha$ embeds in
$C_{\Sigma_\Psi}(a\downarrow_\Psi)=\<a\downarrow_\Psi>$, which
contradicts (e). Therefore $a\downarrow_\Psi=1_\Psi$. If (d1) holds,
then the same argument shows that $a$ fixes $\Omega$ pointwise, whence
$a=1$. On the other hand if (d2) holds, then for any $\beta\in
\Psi$,  $a$ fixes $\beta$ so $a$ acts on $G_\beta$, and (d2) gives
$[a,G_\beta]=1$. In this case $C_G(a)$ contains $\<G_\alpha, G_\beta>$,
which equals $G$ by (a), and again $a=1$.   
\end{proof}

\begin{lemma}[cf. \cite{AlperinGorenstein:VanishingCoho}]
\label{lemmafour}
Suppose that $V$ is an $\F_pG$-module. Let $x\in G$ be a $p'$-element
such that $C_V(x)=0$. Put a graph structure on $x^G$ by joining $x$
and $y$ if and only if $\langle x\rangle$ and $\langle y \rangle$
normalize each other. Let $C$ be the connected component of $x$. If
$G$ normalizes $C$, or if $G=\langle C_G(y)\,|\,y\in C\rangle$, or if $x\in
O_{p'}(G)$, then $H^1(G,V)$ is trivial.
\end{lemma}

\begin{proof}
  It suffices to show that if $W$ is an $\F_pG$-module containing $V$
  with $\dim W=\dim V+1$ and $G$ centralizing $W/V$, then $W=V\oplus
  C_W(G)$. Since $x$ is a $p'$-element and $C_V(x)=0$. we have
  $W=V\oplus W_y$ for each $y\in C$, where we set $W_y=C_W(y)$. Then
  whenever $y,z\in C$ are connected, $z$ normalizes $W_y$ and so
  $W_y=W_z$. Hence $W_y=W_x$ for all $y\in C$. Hence $W_x$ is 
  normalized and then centralized by $N_G(C)$, as well as by $\langle
  N_G(\langle y)\rangle \,|\,y\in C\rangle$. Obviously $C\subseteq
  N_G(C)$ so the result follows. 
\end{proof}

\begin{corollary}
\label{somehones}
  Let $G=F_1$ or $Co_2$. Let $z$ be a $2$-central
  involution of $G$, $C=C_G(z)$, and $Q=O_2(C)$. Then 
  $H^1(C/Q,Q/\<z>)$ is trivial. Consequently, with Lemma
  $\ref{lemmaone}b$, $C_{\aut(C)}(T)=1$, for $T\in\Syl_2(G)$. 
\end{corollary}

\begin{proof}
  Set $H_z=C_G(z)/O_2(C_G(z))$ and
  $V_z=O_2(C_G(z))/\<z>$. In these cases,
  $(C_G(z)/O_2(C_G(z)),\dim V)=(Co_1,24)$ or $(Sp_6(2), 8)$
  and $V=\Lambda/2\Lambda$ or the spin module,
  respectively. In either case $H_z$ possesses an element
  $x$ of order $3$ such that $C_V(x)$ is trivial. We let $C$
  be the connected component as defined in Lemma
  \ref{lemmathree}, and $N=N_{H_z}(C)$. By Lemma
  \ref{lemmathree}t suffices to prove that $N=H_z$.

  If $G=Co_2$, then $\dim([x,V_{nat}])=2$ for the natural
  $6$-dimensional module $V_{nat}$ for $Sp_6(2)$. But in that case $x$
  has an $H_z$-conjugate $y$ such that $[x,V_{nat}]\perp[y,V_{nat}]$,
  and then $N\ge\<C_{H_z}(x),C_{H_z}(y)>\ge
  \<E(C_{H_z}(x)),E(C_{H_z}(y))>=H_z$, as required. 

  If $G=F_1$, then $H_z\cong Co_1$ and $N_{H_z}(\<x>)\cong
  3Suz2$, and there is $P\in\Syl_3(H_z)$ such that
  $A:=J(P)\cong E_{3^6}$, $x\in A=C_{H_z}(A)$, and
  $\Aut_{H_z}(A)\cong 2M_{12}$. Moreover, $N_{H_z}(A)\cap
  N_{H_z}(\<x>)\cong Z_2\times M_{11}$. Since
  $\<x^{N_{H_z}(A)}>=A$ is abelian, $N_{H_z}(A)\le
  N$. Indeed $N_{H_z}(A_0)\le N$ for any
  $A_0\le A$ such that $C\cap A\ne\emptyset$. Such an $A_0$
  of order $9$ exists with $O^2(N_G(A_0))=O^2(C_G(a))\cong
  3^2U_4(3)$ for every $a\in A_0^\#-C$ \cite[Table
  5.3l]{I:A}. We fix $A_0$. Then for all $a\in A_0^\#$,
  either $\<a>\in C$ or $A_0\nor N_{H_z}(\<a>)$. Hence
  setting $\Gamma=\<N_{H_z}(\<a>)\,|\,a\in A_0^\#>$, we have
  $\Gamma\le N$. We also fix $a\in  A_0-C$.  
  
  We choose any elements $x_i\in O^2(C_{H_z}(a))\cong
  3^2U_4(3)$ of order $i$ for $i=5, 7$. Then $C_{H_z}(x_i)$
  contains $A_0$, and hence exactly four elements of $C$. Therefore
  $x_5$ and $x_7$ are of class $5B$ and $7B$, respectively
  \cite[Table 5.3l]{I:A}, and it follows quickly that
  $N_{H_z}(\<x>_i)\le \Gamma\le N$. Next, let $t\in H_z$ be
  any $2$-central involution centralizing any
  element $c\in C$; such elements exist, with $C_{H_z}(ct)$
  an extension of $2^{1+6}_-$ by $\Omega_6^-(2)$ and
  $C_{H_z}(t)$ an extension of $2^{1+8}_+$ by
  $\Omega_8^+(2)$. Thus there is a conjugate $c'\in
  c^{C_{H_z}(t)}$ of $c$ such that $[c,c']=1$ and
  $C_{H_z}(t)\le \<C_{H_z}(c),C_{H_z}(c')>$. Hence
  $C_{H_z}(t)\le\Gamma\le N$. 

  Now $N\ge N_{H_z}(X)$ for $X=A,\<x_5>,\<x_7>,$ and $\<t>$,
  as well as $\<c>$ for any $c\in C$. Together these imply
  that $|H_z:N|$ divides $5\cdot23$. But also
  $|H_z:N|\equiv1\mod{33}$, since $N\ge N(A)\ge N(P)$ and
  Sylow $11$-normalizers of $N_G\<x>\le N$ are Sylow
  $11$-normalizers in $G$. Therefore $|H_z:N|=1$ and the
  proof is complete.
\end{proof}

\section{Group by Group Calculations}

\vskip-5pt
\grp{M_{11}}[Conway:ThreeLec]

Since $G$ is sharply quadruply transitive on $11$ letters, the
stabilizer $G_{\alpha\beta}$ of two points is
sharply doubly transitive -- i.e., a Frobenius group -- of order
$9.8$. Then $M:=N_G(G_{\alpha\beta})$ satisfies $|M|=2|G_{\alpha}|$
and $M$ is easily seen to be a Sylow $3$-normalizer in $G$. Since
$|G:M|$ has no nontrivial divisors that are congruent to $1\pmod3$,
$M$ is maximal in $G$. By a Frattini argument, 
$$
\wt A=\wt N(M).
$$
$M$ is complete ($\Aut(M)=\Inn(M)\cong M$) so $N_G(M)=MC_G(M)$ and
$$
\wt N(M)=\wt C(M).
$$ 

Since $T\le M$, $C_A(M)\le C_A(T)$. Let $\alpha\in C_A(M)$. Then $\alpha$
maps into $C_{\Aut(C)}(T)$. But $C_G(Z)\cong GL_2(3)$. By Lemma
\ref{lemmaone}b, $\alpha$ centralizes $C_G(Z)$. Since $M$ is maximal in
$G$, $G=\<C_G(Z),M>$ is centralized by $\alpha$, so $\alpha=1$,
proving that

\vskip-3pt
\finaldisplay{$\wt C(M)=1$.}

\grp{M_{12}}[Conway:ThreeLec]

We show first that $G$ has at most two conjugacy classes of
$M_{11}$-subgroups.

Suppose that $H$ and $K$ are nonconjugate $M_{11}$-subgroups of
$G$. Then $K$ has no fixed point on $G/H$. Since
$|G:H|=|G:K|=12$ and $K$ has an element of order $11$, $HK=G$.
Therefore $|K:H\cap K|=12$. Consequently $H$ and $K$ share no Sylow
$3$-subgroup. We have proved that $M_{11}$-subgroups of $G$ sharing a
Sylow $3$-subgroup are conjugate, as are any two $M_{11}$-subgroups
whose product is not $G$. 

Now $H$ has a Sylow $3$-subgroup of order $3^2$, all of whose
nonidentity elements are fused in $H$. But in $G$, which has a Sylow
$3$-subgroup of type $3^{1+2}$, there are at most $2$ (exactly two, in
fact) conjugacy classes of subgroups of order $3^2$ all of whose
nonidentity elements are fused in $G$. By the previous paragraph,
there are at most two conjugacy classes of $M_{11}$ subgroups in
$G$. Thus
$$
|\wt A:\wt N(M)|\le 2.
$$
Since $M\cong M_{11}$ and we have just seen that $\Out(M_{11})=1$,
$$ 
\wt N(M)=\wt C(M).
$$ 
$M$ is doubly transitive, hence primitive, on $G/M-\{M\}$. By Lemma
\ref{lemmathree}(d1),
$$ 
\wt C(M)=1.
$$
Finally, for any $\alpha\in \wt C(T)$, $T_M\le M\cap M^\alpha$,
whence $|M:M\cap M^\alpha|$ is odd. By the second paragraph,
$M^\alpha$ is $G$-conjugate to $M$. Thus, $\alpha\in GN_A(M)$, and so
\finaldisplay{$\wt C(T)\le \wt N(M)=\wt C(M)=\wt 1$.}
\medskip

\begin{exnote}
  A subgroup $H=M_{12}.2$ with $F^*(H)=M_{12}$ is visible in
  $M_{24}$ as the stabilizer of a decomposition of the
  underlying $24$-element set into two complementary
  dodecads of the Golay code, proving that the bound $|\wt
  A|\le 2$ is sharp and exhibiting both of the
  quasi-equivalent transitive actions of $M_{12}$ on $12$
  points \cite{Conway:Co}.
\end{exnote}

\grp{M_{22}}[Conway:ThreeLec]

Since $M_{24}$ is quintuply transitive on $24$
points and preserves a Steiner system $S(5,8,24)$, its three-point
stabilizer $M_{21}$ is doubly transitive on $22$ points and preserves
a Steiner system $S(2,5,21)$, i.e. the projective plane of order $4$. 

The maximal parabolic subgroup $P=M_{20}\cong ASL_2(4)$ of $M_{21}\cong L_3(4)$
has no faithful permutation representation of degree less than
$16$. (A point stabilizer $P_\alpha$ would satisfy $1<O_2(P)\cap
P_\alpha<O_2(P)$ so would be reducible on $O_2(P)$, contain no element
of order $5$, then be of index $10$; but $3$-elements of $P$ leave no
hyperplane of $O_2(P)$ invariant.)  As a result, if $M_{21}$ acts on
any set of cardinality $22$, it must be transitive or have a fixed
point; the latter must in fact hold since $11$ does not divide
$|M_{21}|$. Consequently all $M_{21}$-subgroups of $M_{22}$ are
$M_{22}$-conjugate and
$$
\wt A=\wt N(M).
$$ 
As $\Out(M)$ is rather large -- of order $12$ -- we bring other
subgroups to bear as well. $M$ has exactly two conjugacy classes of
(maximal parabolic) subgroups isomorphic to $2^4A_5$, represented by
$N_0$ and $N_1$, say, and such that if we put $Q_i=O_2(N_i)$, then
$N_G(Q_0)/Q_0\cong A_6$ and $N_G(Q_1)/Q_1\cong \Sigma_5$. Therefore
$Q_0$ and $Q_1$ are not $A$-conjugate.  Thus if we set $N_{Q_0}=N_A(M)\cap
N_A(Q_0)$, we have
$$ 
\wt N(M)=\wt N_{Q_0} \text{ and } \Aut_G(Q_0)\cong A_6.
$$ 

Set $C_{Q_0}=C_{N_{Q_0}}(Q_0)$ and $B=\Aut_G(Q_0)\Aut_{N_{Q_0}}(Q_0)$. As $G\nor A$, we
have $\Aut_G(Q_0)\nor B$.  This forces\footnote{$\Aut(Q_0)\cong
GL_4(2)\cong A_8$. Then $\Aut(Q_0)$ and $\Aut_G(Q_0)\cong A_6$ share a
Sylow $3$-subgroup $P$. One calculates $|N_{\aut_G(Q_0)}(P):P|=4$ and
$|N_{\aut(Q_0)}(P):P|=8$. By a Frattini argument $B\le
\Aut_G(Q_0)N_{\aut(Q_0)}(P)$, so $|B:\Aut_G(Q_0)|\le
|N_{\aut(Q_0)}(P):N_{\aut_G(Q_0)}(P)|=2$.} $|B:\Aut_G(Q_0)|\le
2$. Moreover $B/\Aut_G(Q_0)\cong \wt N_{Q_0}/\wt C_{Q_0}$, so 
$$ 
|\wt N_{Q_0}/\wt C_{Q_0}|\le 2.
$$ 

But $\Aut(M)\cong P\Gamma L_3(4)\<\tau>$, where $\tau:D\mapsto
(D^T)\inv$ for projective images $D$ of $3\times 3$ matrices. By a
direct calculation, we see that $C_{\aut(M)}(Q_0)=1$, so
$$ 
\wt C_{Q_0}=\wt C(M).
$$ 

$M$ is doubly transitive, hence primitive, on $G/M-\{M\}$. By Lemma
\ref{lemmathree} (condition (d1) holds),
$$ 
\wt C(M)=1.
$$ 

Finally, let $\beta\in C_A(T)$. Then $\beta$ centralizes $Q_0$ and the
other $E_{2^4}$-subgroup $Q_1$ of $T_M\in\Syl_2(M)$. We have
$Q_i=F^*(N_G(Q_i))$, $i=0,1$, so $\beta$ centralizes
$N_G(Q_i)/Q_i$. Thus $\beta$ normalizes every overgroup of $Q_i$ in
$N_G(Q_i)$. In particular $\beta$ normalizes $\<N_M(Q_0),N_M(Q_1)>=M$.
Thus $\beta\in C_{Q_0}$ and 

\finaldisplay{$ 
\wt C(T)=\wt C_{Q_0}=\wt C(M)=1.
$}
\medskip
\begin{exnote}
 The setwise stabilizer of a two-point set in the
natural action of $M_{24}$ on $24$ letters is a group $H=M_{22}.2$
with $F^*(H)=M_{22}$, so the bound $|\Out(M_{22})|\le 2$ is sharp.

Since $\wt A=\wt N_{Q_0}$, $N_{\aut(M_{22})}(M)=M\<\sigma>$, where
$\sigma$ is a field automorphism of $M\cong L_3(4)$. Hence
$\Aut_{M_{22}}(Q_1)=\Aut_{\aut(M_{22})}(Q_1)$ and
$O_2(N_{\aut(M_{22})}(Q_1)\cong E_{2^5}$. On the other hand as $\wt
N_{Q_0}/\wt C_{Q_0}$ has order $2$, $\Aut_{\aut(M_{22})}(Q_0)\cong
N_{A_8}(A_6)\cong \Sigma_6$. 
\end{exnote}

\grp{M_{23}}[Conway:ThreeLec]

Since a Sylow $2$-subgroup of $M_{21}$ has no faithful
permutation representation of degree less than $16$, the same is true
for $M_{22}$. Nor can $M_{22}$, by its order, act transitively on $23$
letters. Therefore as in the $M_{22}$ argument above, any subgroup of
$G$ isomorphic to $M\cong M_{22}$ has a fixed point on $G/M$, so is
conjugate to $M$, and
$$
\wt A=\wt N(M).
$$

Let $Q_i\le M$, $i=0,1$, be the same subgroups as in the $M_{22}$
analysis, so that $N_i:=N_M(Q_i)$ satisfies $F^*(N_i)=Q_i$,
$N_0/Q_0\cong A_6$, and $N_1/Q_1\cong\Sigma_5$. Then
$N_i^*:=N_{M_{23}}(Q_i)$ satisfies $F^*(N_i^*)=Q_i$ and $\Aut_{M_{23}}(Q_i)\cong
A_7$ or $\Sigma_6$ according as $i=0$ or $1$. Since any subgroup of
$\Aut(Q_0)\cong A_8$ isomorphic to $A_7$ is self-normalizing in $A_8$,
and $\Aut_{M_{23}}(Q_0)\nor \Aut_A(Q_0)$, we have
$\Aut_A(Q_0)=\Aut_G(Q_0)$. But we saw in the $M_{22}$ analysis that
$\Aut_{\aut(M_{22})}(Q_0)\cong\Sigma_6$, which does not embed in
$A_7$. Consequently $\Aut_A(M)=\Inn(M)$, so 
$$ 
\wt N(M)=\wt C(M).
$$ 

$M$ is doubly transitive, hence primitive, on $G/M-\{M\}$. By Lemma
\ref{lemmathree} (condition (d1) holds),

\finaldisplay{$\wt C(M)=1$.}

\grp{M_{24}}[Conway:ThreeLec]

Let $\alpha\in A$ and $H=M^{\alpha}$, and suppose by way of contradiction
that $H$ is not $G$-conjugate to $M$. Fix a series
$G>M=M_{23}>M_{22}>M_{21}>M_{20}$ and let $H_n=H\cap M_n$,
$n=20,21,22,23$. We claim that $|M_n:H_n|=24$ for each such $n$. It is
sufficient to show by descending induction that
$H_n$ is transitive on $M_n/M_{n-1}$ for $n\ge 21$. Since $H$ has no
fixed points on $G/M$ and contains an element of order $23$, the claim
holds for $n=24$. For other values of $n$, the orbit
lengths of $H_n$ on $M_n/M_{n-1}$ are restricted by the  conditions $H_n$,
$n\ge 21$, contains elements $g_p$ of orders $p=5$, $7$, $11$, and $23$, 
with $4$, $3$, $2$, and $1$ fixed points respectively on $G/M$. Thus,
$g_{23}$ implies transitivity for $n=23$; $g_{11}$ and $g_7$ for
$n=22$; and $g_7$ and $g_5$ for $n=21$. Now $M_{20}\cong
ASL_2(4)$ and $|M_{20}:H_{20}|=24$. But no such subgroup $H_{20}$ of
$M_{20}$ exists; we would have $H_{20}\cap O_2(M_{20})\ne 1$ and $5$
divides $|H_{20}|$, whence $O_2(M_{20})\le H_{20}$ and
$|M_{20}:H_{20}|=12$, contradiction. 

Therefore $H$ is $G$-conjugate to $M$, and 
$$
\wt A=\wt N(M).
$$

Since $\Out(M)=1$, 
$$ 
\wt N(M)=\wt C(M).
$$ 

$M$ is doubly transitive, hence primitive, on $G/M-\{M\}$. By Lemma
\ref{lemmathree} (condition (d1) holds),

\finaldisplay{$\wt C(M)=1$.}

\grp{HS}[HigmanSims:HS]

Let $\alpha\in \Aut(G)$, set $H=M^\alpha$, and assume
that $H$ is not $G$-conjugate to $M$. 

$G$ is a rank $3$ permutation group on $M$, and the $M$-orbits on
$G/M-\{M\}$ are of orders $22$ and $77$. Since $1+22$ does not divide
$|G:M|=100$, this action of $G$ on $100$ points is primitive.
Elements of $H$ of order $7$ or $11$ have $G$-conjugates
$g_7,g_{11}\in M$ by Sylow's Theorem. Then $g_7$ has one fixed point
on the $22$-orbit and none on the $77$-orbit; $g_{11}$ has no fixed
points on either of these orbits.  Hence $g_7$ and $g_{11}$ fix two
and one points, respectively, of $G/M$. Therefore there is an
$H$-orbit $\Psi$ on $G/M$ such that $|\Psi|\equiv1\pmod 11$. Since
$|\Psi|$ divides $|H|$ and $|\Psi|\equiv0$, $1$, or $2\pmod7$, the
only possibility is $|\Psi|=56$.

However, $M_{22}$ has no transitive action of degree $56$. For if it
did, then a point stabilizer $K$ in $M_{22}$ would have order
$2^4.3^2.5.11$ and contain a Sylow $p$-normalizer of $M_{22}$ for both
$p=5$ and $p=11$. Hence for any $g\in M_{22}-K$, $K_0:=K\cap K^g$
would have order dividing, and then equaling, $2^4.3^2$. No $3$-local
subgroup of $M_{22}$ contains a group of order $16$, so
$F^*(K_0)=O_2(K_0)$, and the only possibility admitting a group of
order $9$ is $O_2(K_0)\cong E_{2^4}$. Therefore $K$, and $K\cap
M_{21}$, would have elementary abelian Sylow $2$-subgroups. But since
$11$ divides $|K|$, $|K\cap M_{21}|=2^a.3^2.5$, where $a=3$ or $4$. No
such subgroup of $M_{21}$ exists apart from $A_6$, which has
nonabelian Sylow $2$-subgroups.

Therefore we have a contradiction, and
$$ 
\wt A=\wt N(M).
$$ 
As $M\cong M_{22}$, 
$$ 
\wt N(M)/\wt C(M)\le\Out(M)\cong Z_2.
$$
We have seen that $C_{\aut(M_{22})}(M_{21})=1$, and the action of $M$
on the $22$-orbit is primitive. Hence condition (d2) of Lemma
\ref{lemmathree} holds, and so
$$
\wt C(M)=1.
$$
Finally let $\beta\in C_A(T)$. Let $Q_0,Q_1\le T_M$ be as in the
$M_{22}$-analysis. Since $m_2(G)=4$ (visible from the $4^3L_3(2)$
$2$-local subgroup), and $N_G(Q_i)$, like $N_M(Q_i)$, is irreducible
on $Q_i$, we have $Q_i=\Omega_1(O_2(N_G(Q_i)))$. Since $N_M(Q_i)$
contains an element of order $5$, either $O_2(N_G(Q_i))=Q_i$ or
$|O_2(N_G(Q_i))|\ge2^8$. But $|T|=2^9$ and $|N_M(Q_i)/Q_i|=2^3$, so
$Q_i=O_2(N_G(Q_i))$. By Lemma \ref{lemmaone}e, $\alpha$ centralizes
$N_G(Q_i)/Q_i$, so normalizes $N_M(Q_i)$, $i=0,1$. As $M=
<N_M(Q_0),N_M(Q_1)>$, $\alpha$ normalizes $M$. By the $M_{22}$
analysis, $[\alpha,M]=1$. Hence $C_G(\alpha)\ge \<M,T>=G$, $\alpha=1$,
and 

\finaldisplay{$\wt C(T)=1$.}
\medskip

\begin{exnote}
The original construction of $HS$ \cite{HigmanSims:HS} constructed a group
$X$ with $F^*(X)=HS$ and $|X:F^*(X)|=2$. The full automorphism group
also appears as $C_{F_5}(t)/\<t>$ for certain involutions $t\in
F_5$. (The group $O_2(C_{\aut(F_5)}(t))$ is cyclic of order $4$.)
\end{exnote}

\grp{J_1}[Janko:J1]

As $T\in\Syl_2(G)$, a Frattini argument gives
$$
\wt A=\wt N(M).
$$
Since $M=N_G(T)$ is complete, 
$$
\wt N(M)=\wt C(M).
$$
Let $\alpha\in C_A(M)$. For any involution $x\in M$, $C_M(x)\cong
Z_2\times A_4$, so the action of $\alpha$ on $C_G(x)\cong Z_2\times
A_5$ lies in $C_{\aut(Z_2\times A_5)}(Z_2\times A_4)=1$. Therefore
$\alpha$ centralizes $H:=\<N_G(T),C_G(x)\,|\,x\in T^\#>$. By the
Bender-Suzuki Theorem, $G$ has no strongly embedded subgroup. But $H$
is strongly embedded in $G$ unless $H=G$, which must therefore be the
case. Then $\alpha=1$ and

\finaldisplay{$\wt C(M)=1$.}
\medskip

\grp{J_2,\ J_3}[Janko:J2J3,HallWales:J2,HigmanMcKay:J3]

For both $G=J_2$ and $G=J_3$, $M=C_G(Z)\cong
2^{1+4}_-\Omega_4^-(2)=C_G(Z)$. As $Z$ is a Sylow center, 
$$
\wt A=\wt N(M).
$$
We have $F=F^*(M)=O_2(M)$. By Lemma \ref{lemmaone}d, $N_A(M)\cap
C_A(F)=C_A(M)$. Therefore as $\Out(F)\cong O_4^-(2)$,
$$
\wt N(M)/\wt C(M)\le \Out(F)/\Omega_4^-(2)\cong Z_2.
$$
Finally, all subgroups of $F$ of order $2$ are $G$-conjugate to
$Z$. By a standard result on $p$-groups (10.11 of \cite{I:G}), there

\noindent
is $U\cong E_{2^2}$ such that $U\le F$ and $U\nor T$. Thus Lemma
\ref{lemmatwo} applies to give its conclusions $(1)$ and $(2)$. Moreover,
in \cite{Janko:J2J3}, Z. Janko characterized the orders of the two groups
$J_2$ and $J_3$ (and much more) by those conclusions $(1)$ and $(2)$,
and the structure of $C(Z)$. Although two groups emerge, neither of
$J_2$ and $J_3$ contains the other, by Lagrange's Theorem. Hence $W=G$
in the language of Lemma \ref{lemmatwo}, and so

\finaldisplay{$
\wt C(M)\le \wt C(T)=1.
$}
\medskip
\begin{exnote}
Constructions of $J_2$, e.g.  as a rank $3$ permutation group
\cite{HallWales:J2},
and of $J_3$ \cite{HigmanMcKay:J3}, construct a non-inner automorphism
as well. $\Aut(J_2)$ is visible in a number of groups, for example as
the fixed points of a non-inner automorphism of
$Suz$. As a pariah, $J_3$ is not obviously on view
elsewhere, and as proved by Griess is not involved in $F_1$.
\end{exnote}

\grp{J_4}[Janko:J4]

$G$ is of characteristic $2$ type, and $Z(T)$ is cyclic. Therefore as
$M=[C,C]$,
$$
\wt A=\wt N(C)=\wt N(M).
$$

$O_2(C)/Z\cong E_{2^{12}}$ is an absolutely irreducible (faithful)
$6$-dimensional module for $O^2(C)/O_2(C)\cong 3M_{22}$ over
$\F_4$. Therefore $\Aut_C(O_2(C))$ is self-normalizing in
$\Aut(O_2(C))$. By Lemma \ref{lemmaone}e,d applied to $C$, $M$, respectively,
$$
\wt N(C)\le \wt C(M).
$$

Let $\alpha\in C(M)$ be of prime order $p$. If $p$ is odd then
$[\alpha,C]=1$ by Lemma \ref{lemmaone}a and then Corollary
\ref{charptype} gives a contradiction. So $p=2$. Choose $x\in C$ of
order $11$ and let $N=N_G(\<x>)$. Then $R:=F^*(N)\cong 11^{1+2}$ and
$N\cap O_{2,3}(M)\cong SL_2(3)$ is absolutely irreducible on
$R$. Therefore $\alpha$ centralizes or inverts $R/\Phi(R)$. Replacing
$\alpha$ by $\alpha z$ if necessary we may assume that $[\alpha,
R]=1$, whence $N\le H:=C_G(\alpha)$ by Lemma \ref{lemmaone}a. As $M\le
H$, we now have $|G:H|$ dividing $2.23.29.31.37.43$. Moreover by
Corollary \ref{charptype}, $H$ is strongly $11$-embedded in $G$, whence
$|G:H|\equiv1\pmod{11^3}$. These conditions imply that $|G:H|=1$, so

\finaldisplay{$\wt C(M)=1$.}

\grp{Co_1}[Conway:Co] 

The action of $G=Co_1$ on the Leech
lattice $\Lambda$ gives an absolutely irreducible embedding $G\le
PSL_{24}({\bf Z})\le PGL_{24}(\Q)$. We identify $Co_1$ with its image. This
is the only complex irreducible projective representation of $G$ of
degree at most $24$. Therefore any automorphism $\alpha\in A$ is
realized as conjugation by some element $a\in N_{GL_{24}(\Q)}(G)$. Let
$V=\Q^{24}$, and let $\Lambda\subseteq V$ be a copy of the Leech
lattice invariant under $Co_0=2Co_1$. Then $\Q\Lambda=V$, and we may
consider $a\in GL(V)$. Hence $\Lambda^a\subseteq V$, and
$\Lambda/\Lambda^a\cap \Lambda$ is finite and
$Co_0$-invariant. However, the representation of $Co_0$ on $\Lambda$
has the property that it remains irreducible $\pmod p$ for every prime
$p$, i.e., $Co_0$ is irreducible on $\Lambda/p\Lambda$ for every prime
$p$. It follows that $\Lambda^a\cap \Lambda=n\Lambda$ for some integer
$n$, and similarly $\Lambda^a\cap \Lambda=m\Lambda^a$ for some
$m$. Therefore modifying $a$ by a rational scalar we may achieve
$\Lambda^a=\Lambda$. By absolute irreducibility there is a unique
bilinear symmetric nondegenerate rational $G$-invariant form $f$ on
$\Lambda$, up to scalars, and there is a unique one which is
unimodular on $\Lambda$. Therefore $a$ preserves $f$, whence
$a\in\Aut(\Lambda)$, which by definition is $Co_0$. Therefore 
\finaldisplay{$\wt A=1.$}

\grp{Co_2}[Conway:Co]

The action of $M/O_2(M)\cong Sp_6(2)$ makes the Frattini quotient of
$O_2(M)$ the spin module, which is absolutely irreducible. By Corollary
\ref{somehones} and Lemma \ref{lemmaone}b,
$$
\wt A=\wt N(M)=\wt C(M)=\wt C(T), 
$$

It remains to show that $\wt C(M)=1$. Let $\alpha\in C(M)$ be of prime
order $p$. As $G$ is of characteristic $2$ type, $p=2$ by Corollary
\ref{charptype}.  Let $t\in T$ be an involution in class $2B$, and
extremal in $T$. Then $F^*(C_G(t))=O_2(C_G(t))$, and $Z(O_2(C_G(t)))$
is the direct product of a trivial module and a natural module for
$C_G(t)/O_2(C_G(t))\cong A_8$. Hence $Z(C_T(t))=\<t>Z$. By Lemma
\ref{lemmaone}b, $\alpha\downarrow
_{C_G(t)}\in\Aut_{Z(C_T(t))}(C_G(t))=\Aut_{\<z>}(C_G(t))$. Replacing
$\alpha$ by $\alpha z$ if necessary we may assume that
$H:=C_G(\alpha)$ contains $\<C,C_G(t)>$. The subgroups of $C$ and
$C_G(t)$ of order $5$ are not $G$-conjugate so a Sylow $5$-subgroup
$F$ of $H$ is not cyclic. In particular some $x\in H$ of order $5$ is
$5$-central in $G$. But $R:=F^*(C_G(x))=O_5(C_G(x))\cong
5^{1+2}$. Since $[\alpha,F]=1$ and $\alpha$ has order $2$, $[\alpha,
R]=1$ and then $N_G(\<x>)\le H$. Thus, $|G:H|$ divides $3^4.11.23$ and
$|G:H|\equiv 1\pmod{5^2.7}$. The last holds because $C$ is strongly
$7$-embedded in $G$ and because $\<x>$ is weakly closed with
$N_G(\<x>)\le H$, with $R\cong 5^{1+2}$. But there are no such numbers
$|G:H|>1$, so $H=G$ and 
\finaldisplay{$
\wt C(M)=1.
$}

\grp{Co_3}[Conway:Co]

In $Co_3$ there are two classes of
involutions, with centralizers $M=2Sp_6(2)$ and $M_1=Z_2\times
M_{12}$, which we may assume are chosen so that $T\in \Syl_2(M)$ and
$TY_1:=T\cap M_1\in\Syl_2(M_1)$. In particular there is a unique class
of $2$-central involutions. Since $Sp_6(2)$ is complete,
$$
\wt A=\wt N(M)=\wt C(M). 
$$
Let $\alpha\in C(T)$. Then $\alpha$ normalizes $M=C_G(Z)$, and so acts
on $M$ like an element of $Z(T)$, as $T/Z$ is self-centralizing in the
$\chev(2)$-group $M/Z$. But $Z(T)=Z\le Z(M)$, so $C(T)=C(M)$. 

Now $\alpha$ acts on $M_1$ and $\alpha$ centralizes $T_1$.
The restriction of $\alpha$ to $M_1$ therefore lies in
$C_{\aut(M_1)}(T_1)=\Aut_Z(M_1)$, as we have seen in the
$M_{12}$-case. Replacing $\alpha$ by $\alpha z$, where
$Z=\<z>$, we may assume that $[\alpha, M_1]=1$. Set
$H=C_G(\alpha)$. Thus $\<M,M_1>\le H$. There is $R\le G$
such that $R=C_G(R)\cong E_{3^5}$ and $N_G(R)=RL$, $L\cong
Z_2\times M_{11}$. Replacing $R$ by a conjugate we may
assume that $L\le M_1$. Then any involution $z\in L$ is
$M_1$-conjugate to $z$, so $H\ge \<C_G(t)\,|\,t\in L,\
t^2=1\ne t>\ge R$. Also $M$ and $M_1$ have non-conjugate
Sylow $5$-subgroups so $5^2$ divides $|H|$. As $|M|$,
$|M_1|$, and $|RL|$ divide $|H|$, $|G:H|$ divides
$5.23=115$. But $R=J(P)$ for $P\in\Syl_3(G)$, so
$|G:H|\equiv1\pmod3$. If $\alpha\ne1$ it follows that
$|G:H|=5.23$. But $\alpha$ acts on $W:=O_5(C_G(x))$ for
$x\in M$ of order $5$, and $W\cong 5^{1+2}$, with
$|C_W(\alpha)|=|H\cap W|=5^2$, $C_G(W)=Z(W)$, and $N_G(W)/W$
of order $48$ and irreducible on $W/Z(W)$. On the one hand,
$|C_W(\alpha)|=|H\cap W|=5^2$ forces $\alpha$ to be a
$5$-element. Then on the other hand, $\alpha$ centralizes
$N_G(W)/W$, which is irreducible on $W/Z(W)$ and so cannot
normalize $C_W(\alpha)$, a contradiction. Therefore
$\alpha=1$ and so \finaldisplay{$ \wt C(M)=1=\wt C(T).  $}

\grp{Mc}[McL:McL]

$M=C_G(Z)$ so 
$$
\wt A=\wt N(M).
$$
$M\cong 2A_8$ so 
$$
|\wt N(M)/\wt C(M)|\le |\Out(M)|=2.
$$ 

Let $\alpha\in C_A(M)$. Choose any $Q\le T$ such that$Q\cong E_{2^4}$
and $N_T(Q)\in\Syl_2(N_G(Q))$. Then $\alpha$ normalizes $Q$ and
$N_G(Q)$, and $\Aut_G(Q)\cong A_7$. Since $A_7$ is maimal in
$\Aut(Q)\cong A_8$, and since $H^1(A_7,Q)=1$, Lemma \ref{lemmaone}c implies
that $[\alpha,N_G(Q)]=1$. Therefore $C_G(\alpha)$ contains
$\<M,N_G(Q)>>M$. But $\<M,N_G(Q)>=G$, so $\alpha=1$ and
\finaldisplay{$\wt C(M)=1.$}

\medskip
\begin{exnote}
$\Aut(Mc)$ is visible in $Co_1$ as the stabilizer $\cdot223$
of a certain triangle in the Leech lattice \cite{Atlas}. It is also
visible in $Ly$ as $N_{Ly}(V)/V$ for a certain subgroup
$V\le Ly$ of order $3$ \cite{Lyons:LyIE}.
\end{exnote}

\grp{O'N}[ONan:ON,Sims:ON,Andrilli:ON,JansenWilson:ON]

Since $M=C_G(Z)$, 
$$
\wt A=\wt N(M).
$$
Also $|\Out(M)|=4$ and $|\Aut(M):\Aut_G(M)|=2$, so 
$$ 
\wt N(M)/\wt C(M)\le \Out(M)/\Out_G(M)\cong Z_2.
$$ 

We have $T\le M$. Let $\alpha\in C_A(M)\le C_A(T)$ and set
$C_\alpha=C_G(\alpha)$. There is $M_1\le G$ such that $T\le M_1$ and
$F^*(M_1)\cong Z_4\times Z_4\times Z_4$. By Lemma \ref{lemmaone}b,
$[\alpha,M_1]=1$. 

We claim that $C_G(t)\le C_\alpha$ for every involution $t\in
C_\alpha$. To do this it is enough to consider only those $t\in
T$. Since $M_1=[M_1,M_1]$, we have $O^2(C_\alpha)=C_\alpha$, and so by
the Thompson Transfer Lemma (15.16 of \cite{I:A}), $t$ has a
$C_\alpha$-conjugate in $E(C)\cong 4L_3(4)$. Then as $L_3(4)$ has only
one class of involutions, $t$ is $C_\alpha$-conjugate into $O_2(M_1)$,
and then $C_\alpha$-conjugate to $z\in Z$. As $C_G(z)\le C_\alpha$,
the claim is proved. By the Bender-Suzuki Theorem, $C_\alpha=G$, so
$\alpha=1$ and we have proved
\finaldisplay{$\wt C(T)\le \wt C(M)=1$.}
\medskip

\begin{exnote}
Sims \cite{Sims:ON} constructed $O'N$ together with
a non-inner automorphism. See also \cite{Andrilli:ON,JansenWilson:ON}. 
\end{exnote}

\grp{Suz}[Suzuki:Suz]

$G$ is a rank $3$ permutation group of degree
$1782=2.3^4.11$ with point stabilizer $M\cong G_2(4)$. The subdegrees
($M$-orbitlengths) are $1+1365+416$, with the two-point stabilizer
corresponding to the $416$ suborbit being isomorphic to $J_2$. The
$1365$-orbit stabilizer in $M$ is a maximal parabolic subgroup $P$ of
$M$. This action of $G$ is primitive as $417$ does not divide $1782$.

$T_M:=M\cap T=C_T(U)$ where $U$ is the unique\footnote{Uniqueness is
visible in $C_G(Z)\cong 2^{1+6}_-O_6^-(2)$.} normal four-subgroup of
$T$. $N_M(U)=O^2(N_G(U))$. So $A=GN_A(T)$ and $N_G(T)$ normalizes
$N_G(U)$ and $N_M(U)$. Let $\beta\in N_A(U)$. If $M^\beta\ne M$, then
$M\cap M^\beta=N_M(U)$, which is a maximal parabolic subgroup of
$M$. Hence $M^\beta$ has an orbit of length $|M:N_M(U)|=1365$ on
$G/M$. The remaining $417$ points fall into one or two orbits since
$G$ has rank $3$ on both $G/M$ and $G/M^\beta$. The length of any
nontrivial orbit of $M^\beta$ of odd degree is a multiple of the index
$1365$ of any maximal parabolic subgroup of $M$, by Tits's
Lemma. Therefore $M^\beta$ has a fixed point on $G/M$, and so is
$G$-conjugate to $M$. So
$$
\wt A=\wt N(M).
$$
Obviously 
$$
\wt N(M)/\wt C(M)\le \Out(G_2(4))\cong Z_2.
$$

$\Out(M)$ is generated by the image of a field automorphism, and it
follows easily that $C_{\aut(M)}(P)=1$. Hence Lemma \ref{lemmathree} applies,
with condition (d2) holding, and we conclude that
$$
\wt C(M)=1.
$$

Finally, we claim that the two
maximal parabolic subgroups $P$ and $P_1$ of $M$ containing $T_U$ are
$C_A(T)$-invariant. We have $P=N_M(U)$ with $U$ the unique normal
four-subgroup of $T$, and by a standard commutator computation in
$T_U$, $P_1=N_M(Q_1)$ where $Q_1=Z_2(T_U)$. As $T_U=C_T(U)$, both
$X:=N_G(U)$ and $X_1:=N_G(Q_1)$ contain $T$. Also $Z\le U\cap Q_1$, so
since $F^*(C)=O_2(C)$, it follows that $Q_0:=F^*(X)=O_2(X)$ and
$Q_1:=F^*(X_1)=O_2(X_1)$. Therefore by Lemma \ref{lemmaone}e, $\alpha$
normalizes any overgroup of $Q_0$ in $X$ as well as any overgroup
of $Q_1$ in $X_1$. Hence to prove the claim, it is enough to show
that $Q_i\le T_U$, $i=0,1$, for then $Q_0\le P$ and $Q_1\le P_1$. 
Now for $i=0$ and $i$, $Z\le Z(T)\le Z(Q_i)$ and for a Cartan subgroup
$H$ of $N_M(T_U)$, $U=\<Z^H>\le Z(Q_i)$, whence $Q_i\le T_U$, as
desired, proving the claim. Therefore $C_A(T)$ normalizes
$\<P,P_1>=M$, and then acts on $M$ as a subgroup of
$C_{\aut(M)}(T_U)=\Aut_U(M)$. Hence 
\finaldisplay{$\wt C(T)\le \wt C(M)=1$.}
\medskip
\begin{exnote}
In $H=Co_1$, there is a subgroup $Y$ of order $3$ such
that $F^*(Y)\cong 3Suz$ and $|Y:F^*(Y)|=2$, so the bound $|\wt A|\le
2$ is sharp.
\end{exnote}

\grp{He}[Held:HeI,Held:HeII]

Here $M\in\Syl_5(G)$ with $M\cong E_{5^2}$. By a Frattini argument
$$
\wt A=\wt N(M).
$$
Moreover $\Aut_G(M)\cong SL_2(3)*Z_4$, so that $Z(\Aut(M))\le
\Aut_G(M)$ and the image $I$ of 
$\Aut_G(M)$ in $\Aut(M)/Z(\Aut(M))\cong PGL_2(5)$ is isomorphic to
$L_2(3)\cong A_4$. As $\Aut_G(M)\nor \Aut_A(M)$, we get 
$$
\wt N(M)/\wt C(M)\le N_{\Aut(M)/Z(\Aut(M))}(I)/I\cong Z_2.
$$
By Lemma \ref{lemmaone}b applied to $X=N_G(M)$,
$$
\wt C(M)=\wt C(N_G(M)).
$$

Now let $\alpha\in C_A(N_G(M))$ and set $G_0=C_G(\alpha)$.  For any $u\in
M^\#$, $E_u:=E(C_G(u))\cong A_5$ and $ME_u\nor N_G(\<u>)$. Then
$N_{E_u}(M\cap E_u)\le N_G(M)\le G_0$. But $C_{\aut
(E_u)}(N_{E_u}(M))=1$, so $E_u\le G_0$, and then $N_G(\<u>)\le
\<E_u,M>\le G_0$. Hence either $G_0=G$ or $G_0$ is strongly $5$-embedded in
$G$. Suppose the latter.

In $G$ there is a subgroup $H\cong Sp_4(4)\cdot2$, which we may take
to contain $M$, by conjugation. If $H\not\le G_0$, then $H\cap G_0$ is
strongly embedded in $H$, contradicting the Bender-Suzuki Theorem,
p.~20 of [GLS4]. So $H\le G_0$.

\noindent

Let $T\in\Syl_2(G)$ with $T_0:=T\cap H\in\Syl_2(H)$. Then $|T_0|=2^9$
so $|T:T_0|=2$. Let $S\in\Syl_3(H)$. Then $C_G(v)\cong 3A_7$
and $C_H(v)\cong Z_3\times A_5$ for all $v\in S^\#$. So
$C_{\aut(C_G(v))}(C_H(v))=1$ for all such $v$, whence $C_G(v)\le
G_0$. Now $|G:G_0|$ divides $2.7^2$ and
$|G:G_0|\equiv1\pmod{5^2}$, so $\alpha=1$ and 
$$
\wt C(N_G(M))=\wt 1.
$$

Finally suppose that $\beta\in C_A(T)$. Now, $T$ is
isomorphic to a Sylow $2$-subgroup of $L_5(2)$. Therefore
$T$ has exactly two subgroups $A_1,A_2$ isomorphic to
$E_{2^6}$. Moreover, $N_i:=N_G(A_i)$ satisfies
$F^*(N_i)=A_i$ and $N_i/A_i\cong 3A_6$; and $H=\<H\cap
N_1,H\cap N_2>$, as $H\cap N_1$ and $H\cap N_2$ contain
distinct maximal parabolic subgroups of $[H,H]$. By Lemma
\ref{lemmaone}b, \finaldisplay{$\wt C(T)\le \wt C(M)=\wt
  1$.}  \medskip

\begin{exnote}
$\Aut(He)$, of order $2.|He|$, is visible in a
$7$-local subgroup of $F_1$. 
\end{exnote}

\grp{Ly}[Lyons:LyIE,Sims:Ly]

$M=3Mc2$ is the normalizer of a subgroup $R$ of order $3$ whose
conjugacy class is uniquely determined by the isomorphism type of
$N_G(R)$. Moreover, $\Aut(M)=\Aut([M,M])$ so $\Out(M)=1$ by the
calculation for $Mc$. Therefore
$$
\wt A=\wt N(M)=\wt C(M).
$$
A straightforward analysis of $M\backslash G/M$ \cite{Lyons:LyIE} shows that
there are five double cosets of uniquely determined cardinalities, and
$M$ is maximal in $G$. Let $\alpha\in C_A(M)$ and take any involution
$t\in M$. Then $\alpha$ acts on $C_G(t)\cong 2A_{11}$ and centralizes
$C_M(t)=(E\times R)\<u>$ where $E\cong 2A_8$ and $u$ is an involution
inverting $R$. Therefore $\Aut(C_G(t))\cong \Sigma_{11}$ and
$C_{\aut(C_G(t))}(C_M(t))=1$. Consequently $\alpha$ centralizes
$C_G(t)$. So $G=\<M,C_G(t)>\le C_G(\alpha)$, $\alpha=1$, and
\finaldisplay{$\wt C(M)=\wt 1$.}

\grp{Ru}[Rudvalis:RuI,Rudvalis:RuII,ConwayWales:Ru]

By Sylow's Theorem $\wt A=\wt C(Z)$. Let $\alpha\in C_A(Z)$
set $N=M^\alpha$, and suppose that $N$ is not $G$-conjugate to $M$,
i.e., has no fixed points on $G/M$. In \cite{Rudvalis:RuI} this is shown to lead
to a character-theoretic contradiction\footnote{A contradiction is
also available by analyzing the orbit lengths of $N$ on $G/M$. By
construction $|MN/M|=1755$ and there must be a second orbit of length
$1755$, the only possibly odd length. These orbits account for all
fixed points of any $5$-element of $N$ on $G/M$. Remaining are $550$
points on which a Sylow $5$-subgroup of $N$ acts semiregularly. $N$
has a subgroup $N_0\cong L_2(5^2)$, which has no such action on $550$
points.} Therefore $M^\alpha\in M^G$ and
$$
\wt A=\wt N(M).
$$

Since $M\cong {^2}F_4(2)$ is complete, 
$$
\wt N(M)=\wt C(M).
$$
Let $\alpha\in C_A(M)$ and $R\in\Syl_3(M)$, and set
$E:=C_G(Z(R))$. Then $E\cong 3A_6$. But $C_{\aut(E)}(R)=1$ so
$[\alpha,E]=1$. $M$ is maximal in $G$ since $G$ has rank $3$ on $G/M$
with subdegrees $1$, $1755$ and $2304$. Also $C_M(Z(R))\cong SU_3(2)$
so $E\not\le M$. Therefore $\alpha=1$ and
$$
\wt C(M)=1.
$$
Let $\beta\in C_A(T)$. We have $|Z|=2$ and $F^*(C)=O_2(C)$, so $\beta$
centralizes $C$ by Lemma \ref{lemmaone}b. Likewise there is a unique $U\nor
T$ with $U\cong E_{2^2}$, and for the same reason $\beta$ acts on
$N_G(U)$ as conjugation by an element of $Z$. Hence some
$\beta'\in\{\beta,\beta z\}$ centralizes $\<C,N_G(U)>$. But
$M=\<C\cap M,N_M(U)>$, these being maximal (parabolic) subgroups of
$M$. Therefore 
\finaldisplay{$\wt C(T)=\wt N(M)=1$.}

\grp{Fi_{22}, Fi_{23}, Fi_{24}}[Fischer:ThreeTransp,Asch:ThreeTranspGps]

We include the nonsimple group $Fi_{24}$ since we use the fact that
$\Out(Fi_{24})=1$ in the calculation of $\Aut(Fi_{24}')$. 

We have $G=Fi_n$, $n=22$, $23$, $24$. Let $M(n)$ be the
corresponding set of $3$-transpositions. Define $Fi_{21}=U_6(2)$, also
a $3$-transposition group with respect to $M(21)$, the set of root
involutions in $Fi_{21}$. We have an exact sequence $1\to \<t>\to M\to
Fi_{n-1}\to 1$ with $t$ an involution in the $3$-transposition
class. Also $G$ has rank $3$ on $G/M$ with subdegrees
$3510=1+693+2816$, $31671=1+3510+28160$, $306936=1+31671+275264$. In
particular $M$ is maximal in $G$. The
class of $3$-transpositions is unique so
$$
\wt A=\wt N(M), \text{ and } \wt N(M)/\wt C(M)\le\Out(M)
$$
with $\Out(M)\cong Z_2$, $Z_2$, and $1$, respectively. 

We show inductively that $C_A(M)=\<t>$. For $n=21$, we take $M$ as the
stabilizer of a root involution, which is a maximal parabolic subgroup
in $Fi_{21}\cong U_6(2)$. Thus $C_A(M)=\<t>$ when $n=21$ (see 2.6.5e of
of \cite{I:A}). In general let $u\in M-\<t>$ be a $3$-transposition,

\noindent
so that $u=t^g$ for some $g\in G$. Let $\alpha\in C_A(M)$. Then
$\alpha$ centralizes $u$, acts on $M^g$ and centralizes $M\cap
M^g=C_{M^g}(t)$. By induction, $\alpha$ acts on $M^g/\<u>$ like an
element of $\<t>$. Consequently $\alpha$ centralizes $O^2(M^g)$. The
same holds for $\alpha t$, which also centralizes $M$. Therefore some
element of $\alpha\<t>$ centralizes $\<M,O^2(M^g)>=G$, so $\alpha\in \<t>$
as claimed.
In particular 
$$
\wt C(M)=1.
$$
Similarly we show that $C_A(T)\le G$. Indeed let $\beta\in C_A(T)$. 
Then $\beta$ centralizes $t$, acts on $M$, and inductively induces an
inner automorphism on $M/\<t>$. Hence for some $x\in M$,
$\beta':=\beta x$ centralizes $M/\<t>$, so centralizes
$[M,M]\<t>=M$. Thus $\beta\in \beta'G=G$, and 
$$
\wt C(T)=1.
$$
Now if $G=Fi_{23}$, then $G$ has there is also a unique class of
involutions $v$ such that $N:=C_G(v)\cong 2^2U_6(2).2$. Observe that
$v$ is $2$-central in $G$. Then $\wt A=\wt N(N)$, $\wt N(N)=\wt C(N)$
since $\Out(N)=1$, and $\wt C(N)\le \wt C(T)=1$. We conclude that
$|\wt A|\le 2$ for $n-22$ and $\wt A=1$ for $n=23,24$.

\qed
\medskip

\begin{exnote}
There is an involution $y\in Fi_{24}'$ such that
$C_{Fi_{24}'}(y)/\<y>$ is an extension of $Fi_{22}$ by a non-inner
automorphism. 
\end{exnote}

\grp{Fi_{24}'}[Fischer:ThreeTransp]

Since $|Fi_{24}:Fi_{24}'|=2$ and $Z(Fi_{24})=1$, the image $\wt A_0$ of $Fi_{24}$
in $\wt A$ is a group of order $2$. Since $Fi_{24}$ is
complete, as we have just seen, 
$$
\wt A_0=N_{\wt A}(\wt A_0).
$$
Thus $|\wt A|/2$ is odd. Let $\alpha\in A$
with $\wt A$ of odd order, and set $A_1=G\<\alpha>$. We show that $\wt
A_1=1$. By Sylow's Theorem, 
$$
\wt A_1=(N_{A_1}(C)){\wt\ }.
$$
Let $Q=F^*(C)$. Then $Q\cong 2^{1+12}_+$ and $F^*(C/Q)\cong
3U_4(3)$. Then $C/Q\le N_{\aut(Q)}(O_3(C/Q))\cong \Gamma U_6(2)$, the
product of $GU_6(2)$ with a graph automorphism. The corresponding
$6$-dimensional representation of $3U_4(3)$ over $\F_4$ is absolutely
irreducible, so $C/Q$ is self-centralizing in $\Aut(Q)$. As
$\Out(U_4(3))$ is a $2$-group and $\wt A_1$ has odd order, $\Aut_G(Q)$
is self-normalizing in $\Aut_{A_1}(Q)$. Therefore 
$$
(N_{A_1}(C)){\wt\ }=(C_{A_1}(C)){\wt\ }=(C_{A_1}(T)){\wt\ }
$$
as in Lemma \ref{lemmaone}e. We may therefore assume that $[\alpha,T]=1$. 

Choose a non-$2$-central involution $t\in T$ such that $t$ is extremal
in $T$. Then $F:=F^*(C_G(t))\cong 2Fi_{22}$ and $\alpha$ maps into
$O^2(C_{\aut(F)}(T\cap F))$. As $T\cap F\in\Syl_2(F)$ we conclude from
the $Fi_{22}$-calculation that $[\alpha,F]=1$ and then as
$|C_G(t):F|=2$, $[\alpha,C_G(t)]=1$.

Now $M\cong Fi_{23}$ with $T_M\in\Syl_2(M)$. We have
$C_{Fi_{24}}(M)=\<z>\cong Z_2$ with $z\not\in Fi_{24}'$. Also, there
exist involutions $t,u\in Z(T_M)$ such that $C_M(t)\cong 2Fi_{22}$ and
$C_M(u)\cong 2^2U_6(2)\cdot2$. Then $\alpha$ centralizes
$<C_M(t),C_M(u)>$, which equals $M$ as $C_M(t)$ is maximal in $M$. The
suborbits of $M$ on $G/M$ are the same as those of
$N_{Fi_{24}}(M)\cong Z_2\times M$ on $Fi_{24}/N_{Fi_{24}}(M)$, and in
particular $M$ is maximal in $G$. As $T\not\le M$, it follows that
$\alpha=1$ and \finaldisplay{$\wt C(T)=(C_{A_1}(C)){\wt\ }=\wt 1$.}

\grp{F_5}[Harada:F5,Harada:AutF5]

There exists a unique conjugacy class $C_3$ of elements $g\in G$ of order $3$ such that $C_G(g)=\<g>\times E(C_G(g))$, with
$E(C_G(g))\cong A_9$, and $N_G(\<g>)=C_G(g)\<u>$, with
$E(C_G(g))\<u>\cong \Sigma_9$.  There is an element $g_1\in C_3$ such
that $C_G(g)\le M$. Choose three disjoint $3$-cycles $g_2,g_3,g_4\in
E(C_G(g_1))$. Then $A=GN_A(\<g_1>)=GN$, where $N$ stabilizes the set
of four subgroups $\<g_i>$, $i=1,\dots,4$ and so normalizes
$M=\<E(C_G(g_i))\,|\, 1\le i\le 4>$. Thus
$$
\wt A=\wt N(M) \text{ and }\wt N(M)/\wt C(M)\le \Out(M)\cong Z_2.
$$
Let $\alpha\in C_A(M)$ and put $C=C_G(\alpha)$.  Let $f\in M$ be a
$5$-cycle. Then $\alpha$ centralizes $C_M(f)=\<f>\times H$, $H\cong
A_7$. From \cite{Atlas} or Table 5.3w of \cite{I:A},

\noindent
$F^*(C_G(f))=\<f>\times I$, $I\cong U_3(5)$. Thus $\alpha$ maps into
$C_{\aut(I)}(H)$, which, however, is trivial\footnote{One way to see
this is to observe by groups orders that $I=HB$, where $B$ is a Borel
subgroup of $I$ containing some Sylow $5$-subgroup of $H$. Note that
Sylow $2$-subgroups of $B$ are cyclic so $|HB|_2>|B|_2$. Since
$\alpha$ centralizes a $5$-element of $B$, $\alpha$ normalizes
$B$. But then $B\ge [\alpha,B]=[\alpha,HB]=[\alpha,I]\nor I$, so
$[\alpha,I]=1$.}. Hence $C\ge\<M,H>$. This implies that $|G:C|$ divides
$2^5.3.5^2.19$. 

Now $C$, like $M$, contains a subgroup $E$ of order $55$. Then
$O_2(C_G(E))\cong Z_2$ is $\alpha$-invariant and so lies in $C$. Hence
a Sylow $11$-normalizer in $C$ has index $1$ or $2$ in a Sylow
$11$-normalizer in $G$. Consequently $|G:C|\equiv1$ or
$2\pmod{11}$. Likewise $M$ contains a Sylow $7$-normalizer of $G$ and
so $|G:C|\equiv1\pmod7$. 

If $\alpha$ is a $5'$-element then it acts on $N_G(Z(R))$,
where $R\in\Syl_5(I)$. But $Z(R)$ is a Sylow $5$-center in
$G$ and $F^*(N_G(Z(R)))\cong 5^{1+4}$ contains $\<f>\times R$
with index $5$. Hence $\alpha$ must centralize
$F^*(N_G(Z(R)))$ and then $\alpha$ centralizes $N_G(Z(R))$,
by Lemma \ref{lemmaone}a. Thus in this case
$|G:C|\equiv1\pmod5$ as well, and $|G:C|$ divides
$2^5.3.19$. There are no such numbers $|G:C|>1$.

Therefore we may assume that $\alpha$ is a $5$-element. Expand
$Q\in\Syl_3(M)$ to $P\in\Syl_3(G)$, and let $N=N_G(Z(P))$. Then
$|P:Q|=3$ and $F^*(N)cong 3^{1+4}$. Hence $\alpha$
centralizes $F^*(N)$ and then $N$, by Lemma \ref{lemmaone}a. Now $|G:C|$
divides $2^5.5^2.19$ and $|G:C|\equiv1\pmod{21}$, and $|G:C|\equiv1$
or $2\pmod{11}$. Again this forces $|G:C|=1$, so 
\finaldisplay{$
\wt C(M)=1.
$}
\medskip

\begin{exnote}
The automorphism group $A=F_5.2$ is involved in $N_{F_1}(D)$,
where $D$ is a subgroup of $F_1$ of order $5$ and class $5A$
\cite{Atlas} or 5.3z in \cite{I:A}. 
\end{exnote}

\grp{F_3}[Thompson:F3]

Since $M=C$ and $T$ has one class of involutions, 
$$
\wt A=\wt N(M).
$$

There are subgroups $E\le D=N_G(E)$ with $E=C_G(E)\cong E_{2^5}$ and
$D/E\cong \Aut(E)$. 

Then $D$ contains a Sylow $2$-subgroup of $G$ and by replacing $D$
by a suitable conjugate we may assume that $Z\in E$, $D\cap M/E\cong
2^4.L_4(2)$ and $D\cap M/O_2(M)\cong A_8\cong L_4(2)$.

We have $F=O_2(M)\cong 2^{1+8}_+$. By Lemma \ref{lemmaone}d, the restriction
mapping $\Aut(M)\to \Aut(F)$ is injective. Furthermore, since $A_9$
contains a Frobenius group of order $9.8$, $M$ acts absolutely
irreducibly on $F/Z$. Therefore $\Aut(M)/\Aut_F(M)$ embeds in
$\Aut(M/F)\cong\Sigma_9$. However, for any $x\in M$ of order $3$
mapping onto a $3$-cycle $\ov x\in\ov M:=M/F\cong A_9$, we have $C_F(x)=Z$. If
$\Aut(M)/\Aut_F(M)\cong\Sigma_9$, then $\ov M$ contains a subgroup
$\<\ov t>\times \ov H$ with $\ov t$ an involution inverting $\ov x$
and $\ov H\cong \Sigma_7$. Since $\ov x$ is fixed-point-free on $F/Z$,
$\<\ov t>$ is free on $F/Z$ and by the $A\times B$-lemma, $\ov H$ acts
faithfully on $C_{F/Z}(\ov t)\cong E_{2^4}$, which is impossible as
$\Sigma_7$ does not embed in $L_4(2)\cong A_8$. Therefore
$\Aut(M)/\Aut_F(M)\cong A_9$ and 
$$
\wt N(M)=\wt C(M).
$$

We set $N:=\<C,D>$ and argue that $C_G(u)\le N$ for all involutions
$u\in N$, whence $N=G$ by the Bender-Suzuki Theorem (p.~20 of
\cite{GLS4}). Since $|G:C\cap D|$ is odd we may assume that $u\in
C\cap D$. Suppose first that $u\not\in F$. Then in $\ov C=C/F\cong
A_9$, $\ov u$ inverts $3$-cycle, which is fixed-point-free on $F/Z$,
so $u$ acts freely on $F/Z$. If $\ov u$ is a root involution then for
some element $v\in C$ of order $5$, $[\ov v,\ov u]=1$ and
$C_G(v)/O_5(C_G(v))\cong SL_2(3)$. Hence $C_G(v)$ contains an element
$w$ of order $4$ such that $\<w^2>=Z$ and $\ov w=\ov u$. By the free
action of $u$ on $F/Z$, $u$ is conjugate to an element of $\<w>$,
which is absurd as $u$ is an involution. The mapping $C/F\to \Aut(\ov
E)$ is an isomorphism, so $\ov u$ is a $2$-central involution in $\ov
C\cong A_9$. Now $\ov u$ is a transvection on $\ov E\cong E_{2^4}$. On
the other hand $\ov u$ acts freely on $\ov F$, as we saw
above. Therefore $|[\ov u,\ov F]\cap \ov E|=8$. But $u$ centralizes
$[u,F]\cap E$. It follows that $u$ induces a transvection on
$E$. Therefore $u$ is $D$-conjugate to an element of $O_2(D\cap
E)=F$. 

As $u^N\cap F\ne\emptyset$, we now may assume that $u\in F$. Since
$C\cap D/F\cong L_4(2)$ has a natural module and its dual as
composition factors on $F/Z$, $C\cap D$ has two orbits on the set of
involutions in $F-Z$: those in $E$, and the rest. But $C$ is
irreducible on $\ov F$ and hence transitive on the involutions of
$F-Z$. As $D$ is transitive on $E^\#$, it follows that $N$ has one
class of involutions; and then as $C\le N$ we conclude that $N=G$, as
asserted. 

Finally let $\alpha\in C_A(M)$. Then $\alpha$ centralizes $C\cap D$ so
normalizes $D$. By Lemma \ref{lemmaone}b, $\alpha\downarrow
D\in\Aut_Z(D)$. Hence for some $\beta\in\alpha Z$,
$C_G(\beta)\ge\<C,D>=N=G$. Therefore 
$\wt C(M)=1$.

\grp{F_2}[Atlas]
\vskip.1pt
Here $M=C_G(t)$, $F^*(M)\cong 2\,{^2}E_6(2)$, and
$|M:F^*(M)|=2$; moreover there exists an involution $u\in
M-F^*(M)$ such that $C_{F^*(M)}(u)\cong F_4(2)\times Z_2$. 
 Only the involutions in $t^G$ have
centralizers isomorphic to $M$, so
$$
\wt A=\wt N(M).
$$
Since $\Out({^2}E_6(2))\cong \Sigma_3$
\cite[Sec. 2.5]{I:A}, $M/\<t>$ is perfect. 

$$
\wt N(M)=\wt C([M,M]/\<t>).
$$
Let $\alpha\in C_A([M,M]/\<t>)$.  We argue first that
$[\alpha,M]=1$ and then that $[\alpha,G]=1$. 
Since $M/\<t>$ is centerless and $[M,M]$ is quasisimple,
$\alpha$ centralizes $M/\<t>$ and $[M,M]$. If $\alpha$ does
not centralize $M$, therefore, we must have
$u^\alpha=tu\in u^G$. However, from \cite[Table 5.3y]{I:A} we see
that of the involutions $u$ and $tu$, one lies in the class
$t^G$ and the other does not, belonging instead to class
$2C$. This contradiction shows that 
$$
\wt N(M)=\wt C(M).
$$

It also shows that we may take $u\in t^G$. Now let
$\alpha\in C(M)$. There is  $v\in M$ of order $3$ such that
$N:=N_{M/\<t>}(\<v>)\cong \Sigma_3\times U_6(2)\cdot2$
\cite[Table 7.3.4]{I:A}. Having a nonsolvable centralizer,
$v$ must belong to class $3A$ \cite[Table 5.3y]{I:A}, and so
$J:=N_G(\<v>)\cong \Sigma_3\times\Aut(Fi_{22})$. Therefore
$\alpha$ acts on $E(J)\cong Fi_{22}$ and centralizes both
$t$ and $E(C_{E(J)}(t))=N^{(\infty)}\cong U_6(2)$. In the
discussion of the case $G=Fi_{22}$, we have argued that (in
the terminology of the current case)
$C_{\aut(E(J))}(N^{(\infty)}\<t>)=\<t>$. Therefore replacing
$\alpha$ by $\alpha t$ if necessary, we may assume that
$[\alpha,E(J)]=1$. Now $t^{E(J)}$ is a class of
$3$-transpositions in $E(J)$. Hence there is $g\in E(J)$
such that $tt^g$ has order $3$ and $E(C_{E(J)}(tt^g))\ne
1$. In particular $C_G(tt^g)$ is nonsolvable, so
$N_G(\<tt^g>)\cong S_3\times\Aut(Fi_{22})$ \cite[Table 5.3y]{I:A}.

Finally consider $H:=C_G(\alpha)$. We have $M\le H$ and
$g\in E(J)\le H$. Then 
$$
\begin{aligned}
M\cap M^g&=C_G(\<t,t^g>)\cong
\Aut (Fi_{22}),\\
|H|&\ge
|MM^g|=\frac{|M|^2}{|M\cap M^g|}=\frac{(4|{^2}E_6(2)|)^2}{2|Fi_{22}|},\\
|G:H|&\le
\frac{|F_2||Fi_{22}|}{8|{^2}E_6(2)|^2}=\frac{3^45^423\cdot31\cdot47}{2^{17}7\cdot17\cdot19}\le6.\\
\end{aligned}
$$
As $G$ is simple and $|G|>6!$, $G=H$, and so
\finaldisplay{$\wt C(M)=1$.}


\grp{F_1}[Griess:F1]

Here $M=C$. By Sylow's Theorem, 
$$
\wt A=\wt N(M).
$$

The action $C/F^*(C)\cong Co_1$ on $F^*(C)/\<z>\cong \Lambda/2\Lambda$
is absolutely irreducible, and $H^1(C/F^*(C),\Lambda/2\Lambda)$ is
trivial by Corollary \ref{somehones}. Therefore by Lemma
\ref{lemmaone}b, 
$$
\wt N(M)=\wt C(M).
$$
In particular $\wt C(T)\le \wt N(M)=\wt C(M)$. 

Let $\alpha\in C(M)$. Let $t\in M$ be a non-$2$-central
involution which is extremal in $T$ and set $H:=C_G(t)$.
Then $\alpha$ acts on $H\cong 2F_2$, and $\alpha$
centralizes $H\cap T=C_T(t)\in\Syl_2(H)$. By the $F_2$
case, $[C_T(t),\alpha]=1$ implies that $[H,\alpha]=1$. 
Hence $C_G(\alpha)$ contains $M$ and $H$.

But it is well-known, and we give an elementary proof below,
that 
\begin{myeqn}
	\label{eq:ghm}
\<M,H>=G.
\end{myeqn}

Thus $\alpha$ will necessarily be trivial
and we will have proved
$$
\wt C(M)=1.
$$

Our argument that $\<M,H>=G$ begins by setting $G_0=\<M,H>$
and observing that $F^*(C_G(\<z,t'>))=F^*(C_M(t))$ is a
$2$-group for any involution $t'$ such that
$[z,t']=1$. Taking $t'=t^g$ we have $[z^{g\inv},t]=1$ and
$F^*\left(C_G\left(\langle z^{g\inv},t\rangle\right)\right)$
is a $2$-group. As $g$ varies, $z^{g\inv}$ varies over all
of $z^G\cap C_G(t)=z^G\cap H$. Therefore any $z'\in z^G\cap
H$ is $H$-conjugate to $z_1$ or $z_2$, these being conjugacy
class representatives such that $C_{H/\<t>}(z_1)\cong
2^{1+22}Co_2$ and $C_{H/<t>}(z_2)\cong 2^{9+16}O_8^+(2)$. To
pull back to $H$, notice that we may take $z_1=z$ in the
first case, and then $t\in O_2(C)$ so $t$ and $tz$ are
$C$-conjugate. In the second case, writing $z_2=z^{g\inv}$
and $t'=t^g\in C$, we have
$C_C(t')=C_G(\<t',z>)=C_G(\<t,z_2>^g)\cong
C_G(\<t,z_2>)=C_H(z_2)$, whence $t'\not\in O_2(C)$ and the
image of $t'$ in $C/O_2(C)\cong Co_1$ is a $2$-central
involution, with $C_{C/O_2(C)}(t'O_2(C))$ is an extension of
a $2$-group by just $\Omega_8^+(2)$. Consequently $z_2$ and
$z_2t$ are fused in $H$. We have proved that there are
exactly two $G$-orbits on the set $P_G$ of all pairs
$(z^h,t^k)$ such that $h,k\in G$ and $[z^h,t^k]=1$, and they
are represented by $(z,t)$ and $(z,t')$.

Since $C\le G_0$, it follows that there are exactly one
or two $G_0$-orbits on the set $P_H$ of all pairs
$(z^h,t^k)$ as above but with $h,k$ restricted to lie in
$G_0$; and representatives of these orbits are $(z,t)$
and (in the two-orbit case) $(z,t')$. We consider these
cases separately.

In the one-orbit case, $t'\not\in t^{G_0}$ and $t^G\cap
C\subseteq O_2(C)$. Therefore $t^G\cap (C\cap H)\subseteq
O_2(C)\cap H= O_2(C\cap H)$. Since $O_2(C\cap
H)/\<z,t>\cong E_{2^{22}}$ is acted on irreducibly by $C\cap
H/O_2(C\cap H)\cong Co_2$, and $t^C\not\le \<t,z>$, it
follows that $\<t^{G_0}\cap C\cap H>=O_2(C\cap H)$. On the
other hand, by the structure of $H/Z(H)\cong F_2$,
there is $t^*\in t^G\cap H$ such that in $C_G(\<t,t^*>)\cong
2^2{^2}E_6(2)$, long root involutions are $H$-conjugate to
$t$. The positive long root subgroups generate a $2$-group
isomorphic to a Sylow $2$-subgroup of $D_4(2)$, which is not
embeddable in the class $2$ group $O_2(C\cap H)$, a
contradiction. 

Therefore we must be in the two-orbit case, so $t^G\cap
C=t^{G_0}\cap C$, that is, $t^G\cap G_0=t^{G_0}$. From this
one quickly gets $t^x\in G_0\iff x\in G_0$ for all $x\in G$,
and one could argue that for any involution $y\in G_0$,
$y^x\in G_0\iff x\in G_0$. This would give $G_0=G$, as
desired, since $G_0$ cannot be strongly embedded in $G$ by
the Bender-Suzuki Theorem. We argue differently,
however. Using \cite[Table 5.3l]{I:A}, we see that in
$C/O_2(C)\cong Co_1$ there exists a subgroup $D\times W$ with $D\cong
D_{10}$ and $W\cong A_5\wr Z_2$; moreover, the only
involutions centralizing an isomorphic copy of $W$ are
$2$-central. We have seen that there is a conjugate $t'\in
t^G$ such that $t'\in C$ and the image of $t'$ in
$C/O_2(C)$ is $2$-central.  Hence using the Baer-Suzuki
theorem, $t'$ inverts an element of $G_0$ of order
$5$. Moreover, $t'\in G_0\cap t^G=t^{G_0}$. Thus, $t$
inverts some $f\in G_0$ of order $5$. Let $t''=tf$. Then in
the action of $G$ on $t^G\times t^G$, the stabilizer of
$(t,t'')$ is $G_{t,t''}=C_G(f,t)$. From \cite[Table
5.3z]{I:A} we see that $|G_{t,t''}|\le
|F_5|=2^{14}3^65^67.11.19$. Note that 
all $G$-conjugate pairs of the form $(t,t^*)$ are actually
$H=C_G(t)$-conjugate. Thus,
$$
\begin{aligned}
|G_0:H|&=|t^{G_0}|\ge|\{t^*\in t^{G_0}\,|\,(t,t^*)\in (t,t'')^G\}|\ge
|H|/|G_{t,t''}|\ge |H|/|F_5|,\\
|G:G_0|&\le
\frac{|G||F_5|}{|H|^2}=\frac{5^37^3.11.13.29.41.59.71}{2^{24}.17.23.31.47}<4.
\end{aligned}
$$

As $G$ is simple, $G=G_0$, as desired.

\bibliographystyle{amsplain}
\bibliography{mybib}

\end{document}